\documentclass[10pt,a4paper]{article}

\usepackage[T1]{fontenc}
\usepackage{lmodern}
\usepackage{amsmath,amssymb,amsfonts,amsthm,amsopn}
\usepackage{graphicx,epstopdf}
\usepackage{pgfplotstable,booktabs,multirow,multicol}
\usepackage[normalem]{ulem}
\usepackage{comment}
\usepackage[numbers,sort&compress]{natbib}
\usepackage{hyperref}
\usepackage[nameinlink,noabbrev]{cleveref}

\DeclareGraphicsExtensions{.pdf,.png,.jpg,.eps}

\theoremstyle{plain}

\theoremstyle{definition}

\theoremstyle{remark}

\crefname{hypothesis}{hypothesis}{hypotheses}
\Crefname{hypothesis}{Hypothesis}{Hypotheses}
\crefname{fact}{fact}{facts}
\Crefname{fact}{Fact}{Facts}

\numberwithin{equation}{section}

\title{Numerical investigation of a homogenized model with effective interface conditions for Stokes flow through porous membranes}
\author{\parbox{0.94\textwidth}{\centering
Jonas Knoch\textsuperscript{1}, Pietro Italiano\textsuperscript{2},
Markus Gahn\textsuperscript{3}, Carsten Gräser\textsuperscript{1},
Maria Neuss-Radu\textsuperscript{1}\\[1ex]
\small\textsuperscript{1}Department Mathematik, Friedrich-Alexander-Universität Erlangen-Nürnberg, Cauerstraße 11, 91058 Erlangen, Germany\\
\small\textsuperscript{2}Institute of Geophysics and Geoinformatics, TU Bergakademie Freiberg, Gustav-Zeuner-Straße 12, 09599 Freiberg, Germany\\
\small\textsuperscript{3}Department Mathematik, Universität Augsburg, Universitätsstraße 14, 86159 Augsburg, Germany\\[1ex]
\small\href{mailto:jonas.knoch@fau.de}{jonas.knoch@fau.de}\quad
\href{mailto:pietro.italiano@geophysik.tu-freiberg.de}{pietro.italiano@geophysik.tu-freiberg.de}\\
\small\href{mailto:markus.gahn@uni-a.de}{markus.gahn@uni-a.de}\quad
\href{mailto:graeser@math.fau.de}{graeser@math.fau.de}\quad
\href{mailto:maria.neuss-radu@math.fau.de}{maria.neuss-radu@math.fau.de}
}}
\date{}

\hypersetup{
  pdftitle={Numerical investigation of a homogenized model with effective interface conditions for Stokes flow through porous membranes},
  pdfauthor={Jonas Knoch, Pietro Italiano, Markus Gahn, Carsten Gräser, Maria Neuss-Radu},
  hidelinks
}

\usepackage{mathtools}
\usepackage{float}
\usepackage{comment}
\usepackage{amssymb}
\usepackage{stmaryrd}
\usepackage{multicol}
\usepackage{enumitem}

\usepackage{listings} 
\definecolor{codegreen}{rgb}{0,0.6,0}
\definecolor{codegray}{rgb}{0.5,0.5,0.5}
\definecolor{codepurple}{rgb}{0.58,0,0.82}
\definecolor{backcolour}{rgb}{0.95,0.95,0.92}
\lstdefinestyle{mystyle}{
    backgroundcolor=\color{backcolour},   
    commentstyle=\color{codegreen},
    keywordstyle=\color{magenta},
    numberstyle=\tiny\color{codegray},
    stringstyle=\color{codepurple},
    basicstyle=\ttfamily\footnotesize,
    breakatwhitespace=false,         
    breaklines=true,                 
    captionpos=b,                    
    keepspaces=true,                 
    numbers=none,                    
    numbersep=5pt,                  
    showspaces=false,                
    showstringspaces=false,
    showtabs=false,                  
    tabsize=2
}
\usepackage[caption=false]{subfig} 

\pgfplotsset{compat=1.18}  

\crefname{figure}{Figure}{Figures}

\usepackage[font=footnotesize]{caption}

\newcommand{\oeps}{\Omega_{\varepsilon}}
\newcommand{\oef}{\Omega_{\varepsilon}^f}
\newcommand{\Oepm}{\Omega_{\varepsilon}^{\pm}}

\newcommand{\Geps}{\Gamma_{\varepsilon}}

\newcommand{\inS}[1]{\mbox{in } & #1}
\newcommand{\onS}[1]{\mbox{on } & #1}

\newcommand{\veps}{v_{\varepsilon}}
\newcommand{\peps}{p_{\varepsilon}}

\newcommand{\eps}{\varepsilon}
\newcommand{\R}{\mathbb{R}}

\newtheorem{remark}{Remark}[section]

\newtheorem{proposition}{Proposition}[section]
\newtheorem{lemma}{Lemma}[section]

\graphicspath{{../}}

\begin{document}
\maketitle

\begin{abstract}
We investigate numerically fluid flow through a thin porous layer, which is modelled as a lower-dimensional interface separating two bulk domains. At this interface, we impose effective transmission conditions involving homogenized coefficients rigorously derived via two-scale methods in previous works. These coefficients are determined by cell problems encoding the microscale properties of the porous layer. In a first step, we analyze qualitatively how the effective coefficients depend on the pore geometry. Based on these computations, we then present numerical simulations of the macroscopic models incorporating the effective interface conditions.
Although replacing the thin porous layer by an effective interface substantially reduces the computational cost, the resulting nonstandard transmission conditions pose additional challenges for the numerical implementation. We employ a Taylor--Hood-type mixed finite element discretization with discontinuous
pressure and tangential velocity across the interface and prove stability as well as error estimates for the resulting scheme. Finally, for microscale geometries that are still computationally resolvable, we compare the results of the macroscopic models with those obtained from direct simulations of the microscopic models.
\end{abstract}

\noindent\textbf{Keywords.} Stokes flow; porous membranes; effective transmission conditions; homogenization; finite element method; numerical simulations.

\medskip
\noindent\textbf{MSC 2020.} 35B30, 76M10, 65N30, 76Z05, 76M50.

\section{Introduction}\label{sec:introduction}
We are concerned with the numerical investigation of transmission models describing fluid flow in two bulk domains $\Omega_\eps^\pm$ separated by a thin porous membrane $\Omega_\eps^M$, see \cref{fig:IntroFig}. Modelling approaches based on first-order principles yield models at the   \textit{microscopic} level, where the geometry of $\Omega_\eps^M$ is resolved and a flow problem is solved in the bulk regions as well as in the pores of the membrane. However, numerical simulations for such \textit{microscopic} models are quite computationally expensive due to the complex structure of the porous membrane. Furthermore, in practical applications, it is often the \textit{macroscopic (effective)} behavior of the flow in the bulk regions that is of interest. For this reason, in many modeling papers, see, for example, \cite{Kolitsi2020, shapiro2015,EarlyArterio}, effective transmission models have been derived phenomenologically by replacing the porous membrane by an interface and formulating effective transmission conditions across this interface. A significant drawback of such models is that the effective parameters in the transmission conditions are not related to the microscopic structure of the membrane or the flow within the membrane pores. Thus, they have to be estimated from the data of every individual application scenario. 

Recently, in \cite{NeussRaduGahn2025, GahnNeussRadu2026} (see also \cite{gahn2025beffective}), transmission models of \textit{micro-macro} type, approximating these microscopic models, have been derived by a rigorous homogenization and dimension reduction approach, see again \cref{fig:IntroFig}. These models consist of two subproblems, a macro (effective) transmission problem formulated on the subdomains $\Omega^\pm$ separated by the interface $\Sigma$, and a micro (cell) problem formulated on the standard periodicity cell $Z_f$ characterizing the microscopic geometry. The effective transmission problem consists of the Stokes equations in the bulk domains coupled by effective transmission conditions at the interface $\Sigma$. These transmission conditions include the continuity of the normal component of the velocity, a relation for the jump in the normal component of the normal stress across the interface, and a Navier-slip-type condition for the tangent component of the normal stress on each side of the interface. The effective parameters in the transmission conditions are computed from solutions of the cell problems, which consist of Stokes problems on the standard cell $Z_f$ with periodic boundary conditions in the directions parallel to the membrane and homogeneous Dirichlet conditions on the boundary of the solid phase. On the top and bottom of the standard cell Dirichlet conditions for the velocity are imposed. We remark that the micro and macro subproblems are one-way coupled through the dependence of the effective coefficients on the solutions to the cell problems. On the one hand, this dependence yields an improved model in comparison to the heuristically derived effective models; on the other hand, however, it leads to a nonstandard transmission problem.
\begin{figure}
    \centering
\includegraphics[width=0.725\linewidth]{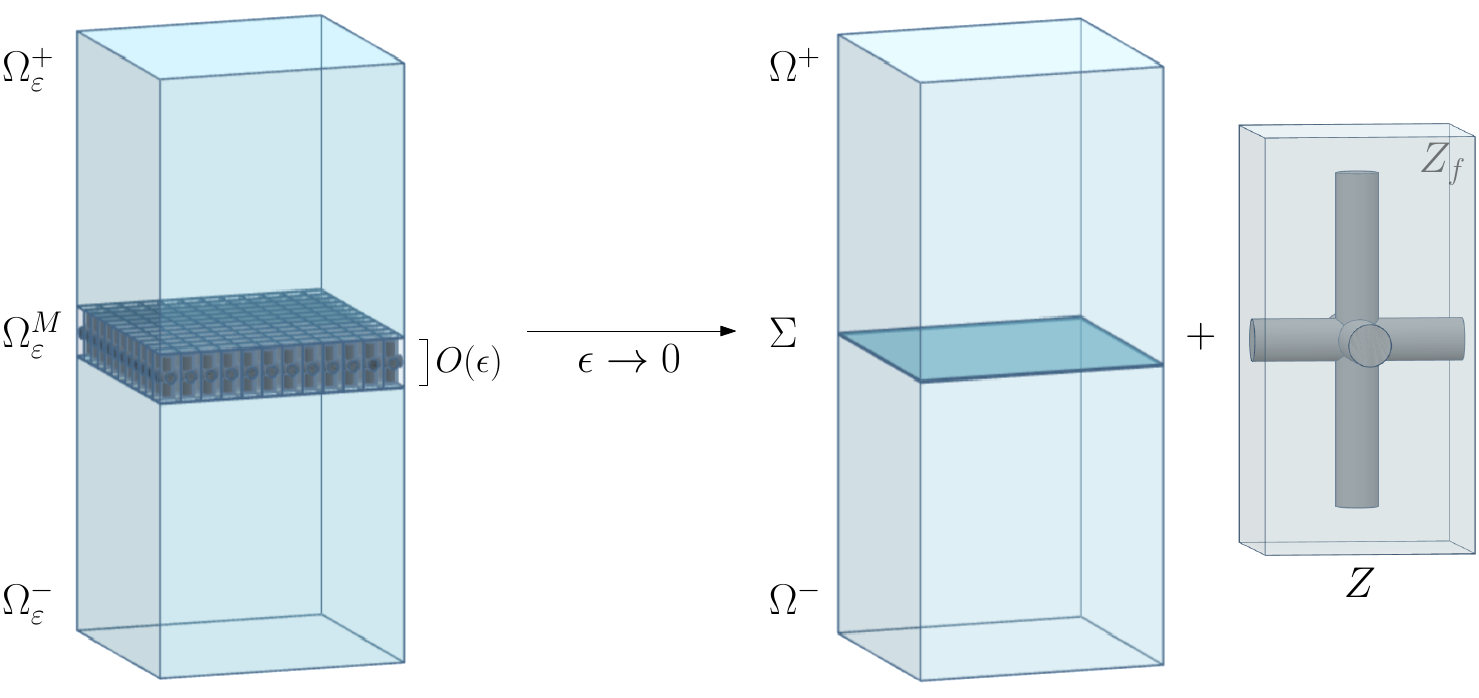}
    \caption{Sketch of the homogenization and dimension reduction approach for the derivation of micro-macro transmission models from microscopic models for flow through porous membranes, together with the reference cell.}
    \label{fig:IntroFig}
\end{figure}

In this paper, we perform the numerical analysis of micro-macro transmission problems described above and develop simulation tools tailored to their specific structure. Furthermore, we perform numerical experiments including different geometries of the microscopic structure, as well as
various boundary conditions at the outer boundary of $\Omega$, which offer detailed information about the behavior of the cell solutions as well as of the effective flow. For both subproblems, the sensitivity of the solutions to changes in the microscopic membrane geometry is studied. 

Micro-macro transmission problems for fluid flow in domains separated by a porous membrane have, e.g., been considered in
\cite{FernandezGerbeauMartin2008}
and \cite{KrierOrlikPanasenkoSteiner2024}. There the cell problems and the effective transmission conditions have been adapted from the literature and differ from those considered in our model. In \cite{FernandezGerbeauMartin2008} the Stokes cell problems are defined on an infinite outer domain and in \cite{KrierOrlikPanasenkoSteiner2024} they are periodic with respect to all coordinate directions in $\R^3$. In our model, the Stokes cell problems are periodic just with respect to the coordinate directions parallel to the membrane and have Dirichlet boundary conditions at the top and bottom of the standard cell.

In both \cite{FernandezGerbeauMartin2008,KrierOrlikPanasenkoSteiner2024}, the fluid velocity is continuous at the interface $\Sigma$, the jump
in the normal component of the normal fluid stress is related to the velocity via a permeability tensor, and the pressure is expected to be discontinuous across $\Sigma$.
Consequently \cite{FernandezGerbeauMartin2008}
develops and analyzes a stabilized mixed finite element discretization with globally continuous velocities and element-wise discontinuous pressure. In the numerical examples it is observed that a classical Taylor--Hood discretization with globally continuous pressure requires severe refinement near $\Sigma$ to provide reasonable approximations, while a modified method with discontinuous pressure across $\Sigma$ performs well. Stability and error bound for the latter method are not discussed. The stabilized method of \cite{FernandezGerbeauMartin2008}
is also used in \cite{KrierOrlikPanasenkoSteiner2024}.
In contrast to the models considered in these works,
the effective (macroscopic) transmission problem obtained here only provides continuity of the normal velocity while the tangential velocity and pressure are in general discontinuous across $\Sigma$. Hence we introduce a novel Taylor--Hood-type mixed discretization where both, the pressure and tangential velocity, are allowed to be
discontinuous across $\Sigma$. The discretization can also be viewed as a coupling of the classical Taylor--Hood discretization for both subdomains where the pressure space is adjusted to the
outer boundary conditions. This coupling view is used to prove stability of the mixed discretization which provides well-posedness and optimal order error bounds by classical saddle point theory which is also observed in numerical examples.

The computation of the effective coefficients requires
to implement the microscopic cell problems with periodic
boundary conditions on a reference cell $Z_f$ which has
a complicated geometry due to the embedded solid structure, while the discretization of the macroscopic transmission problem requires to work with finite element product spaces which are defined on the subdomains
$\Omega^+$ and $\Omega^-$. The implementation of the discretization for both problems was carried out in the modular \texttt{C++} framework \textit{DUNE} \cite{DUNE:2021,DUNE1}. We in particular make use of its possibility to construct complicated product spaces \cite{EngwerEtAl2025} and to implement assemblers by defining the respective variational forms \cite{fufem}.
Support for finite element spaces and variational forms on subdomains and interfaces was added to the respective \textit{DUNE} modules only recently and its implementation was pushed forward by the present application.

To illustrate the influence of the porous layer on the fluid flow, we consider two different cell geometries, a symmetric one, where the solid phase consists of the union of cylinders, and an asymmetric one, which contains an inclined segment. Furthermore, we select four types of boundary conditions at the outer boundary of $\Omega$ driving the flow, and examine the different behavior of the solutions of the transmission problem. It turns out that the influence of the microstructure is captured by the effective transmission conditions, e.g., an asymmetry in the microscopic geometry causes a deflection of the effective velocity. In case of the symmetric cell geometry, we perform a sensitivity study with respect to variations in the height and the radius of the cylinders for the cell solutions, the effective coefficients and the solutions to the effective transmission problem. One interesting feature is that for the case when the height of the vertical cylinder approaches its maximal value (and thus the distance between the solid skeleton and the bulk regions becomes very small), the coefficients explode and the effective velocity at the interface $\Sigma$ converges to zero, i.e., the membrane becomes impermeable. On the other hand, the limit case when the solid skeleton touches the bulk regions has been treated in \cite{NeussRaduGahn2025} and a zero velocity at $\Sigma$ was found. Thus,  the numerical results indicate a continuous dependence of the velocity on the microscopic geometry, a result which has not been proved analytically so far. Last but not least, we perform a numerical validation of the effective model by comparison with solutions to the microscopic model for a specific scenario.

Further micro-macro transmission models obtained by homogenization and dimension reduction for (nonlinear) reaction-diffusion problems can be found in
\cite{NeussRaduLudwig_2010,FreudenbergEden2024}. In both contributions, the cell problem cannot be decoupled from the effective model resulting in strongly (two-way) coupled micro-macro models. 

Our paper is structured as follows. In Section~\ref{sec:mathematicalModel}, we formulate the micro-macro model comprising the effective transmission model, the Stokes cell problems and the effective coefficients. Furthermore, we define adequate function spaces and bilinear forms, which help us to write the weak formulation of the effective transmission model in the form of a saddle point problem, which is solved by using linear existence theory. The numerical computation of cell solutions and effective coefficients, including mesh generation and implementation of the cell problems, a numerical convergence study, and a sensitivity study related to geometric parameters of the microstructure, is performed in Section~\ref{sec:numerical_cellSolCoeff}.
In Section~\ref{sec:macroscopicSimulations}, we introduce
a finite element discretization for the macroscopic effective transmission model, show well-posedness and convergence, perform numerical convergence studies and
numerical experiments for different boundary conditions,
and carry out a parameter study with respect to the 
geometric parameters. Finally, microscopic and effective solutions are compared numerically for a specific computational scenario. A discussion of the obtained results also related to future developments is presented in Section~\ref{sec:conclusionoutlook}.

\section{Statement of the mathematical model}\label{sec:mathematicalModel}
In this section, we give the precise formulation of the mathematical model that will be treated numerically in the following sections. 
Let $\Omega \coloneqq \Sigma\times(-H,H)\subset \mathbb{R}^3$ with $H>0$ and $\Sigma = (a,b)\subset \mathbb{R}^2$, with $a,b \in \mathbb{Z}^2$ and $a_i<b_i$ for $i=1,2$. The subdomains $\Omega^\pm$ are defined by
$$
\Omega^+ \coloneqq \Sigma\times(0,H), \hspace{1.2cm} \Omega^- \coloneqq \Sigma\times(-H,0).
$$ 
Let $\partial_N\Omega \subset \partial \Omega$ be the boundary part of $\Omega$ where normal stress boundary conditions are formulated. We assume that $\partial_N \Omega$ is relatively open (in $\partial \Omega$), in particular for $\partial_N \Omega \neq \emptyset$ we have $|\partial_N \Omega|>0$, where here $|\cdot|$ denotes the $2$-dimensional Hausdorff-measure. Let  $\partial_D \Omega = \partial \Omega \setminus \overline{\partial_N \Omega}$ be the Dirichlet boundary part. We denote by 
\begin{equation}
    \label{Def_boundary_parts}
\partial_D \Omega^\pm = \partial_D \Omega \cap \partial \Omega^\pm, \hspace{1.2cm} \partial_N \Omega^\pm = \partial_N \Omega \cap \partial \Omega^\pm,     
\end{equation}
and by $\nu^\pm$ the outer unit normal to $\partial \Omega^\pm$.
In addition to the domain $\Omega$, we define the periodicity cell $Z$ 
$$
Z \coloneqq Y\times(-1,1) \coloneqq (0,1)^2\times(-1,1),
$$
with top and bottom denoted by 
$$
S^\pm \coloneqq Y\times\{\pm 1\}.
$$
It consists of a solid part $Z_s\subset Z$ (also referred to as \textit{obstacle} or \textit{skeleton}) and a fluid (pore) part $Z_f = Z\setminus \overline{Z_s}$, separated by the interface  $\Gamma = \text{int}(\overline{Z_f} \cap \overline{Z_s})$. Here, by $\text{int}(M)$ we denote the interior of a set $M\subset \mathbb{R}^n$. Hence, we have 
$$
Z = Z_s \cup Z_f \cup \Gamma.
$$
One example of a periodicity cell to be considered in this paper is given in \cref{fig:IntroFig}, where the solid phase $Z_s$ is represented as a cross of cylinders.

Throughout the paper we assume that $\partial Z_s \cap S^\pm = \emptyset$, i.e., the solid skeleton is not touching the upper and lower boundary of the cell.
The other case was also treated in \cite[Section 6.3]{NeussRaduGahn2025} and leads to an impermeable membrane for the fluid flow. A discussion of the case in which the distance between the solid skeleton and the top/bottom of $Z$ approaches zero is given in Remark~\ref{remark_divergence}

We further assume that $Z_f$ and $Z_s$ are open, connected with Lipschitz-boundary,
and the lateral boundary is $Y$-periodic, which means that for $i\in\{1,2\}$ and $*\in\{s,f\}$ 
$$
\partial Z_{*} \cap \{y_i=0\}+e_i = \partial Z_{*} \cap \{y_i=1\}.
$$

\subsection{Notations and function spaces}
For $i\in {1,2,3}$, we denote by $e_i \in \R^3$ the standard unit vectors. For function spaces, we use the index $\#$ to indicate functions which are periodic with respect to the first two variables. Hereby, we will consider periodic functions defined on the standard cell $Z$ (which we call \textit{$Y$-periodic}) and periodic functions on the subdomains $\Omega^\pm$ (which we call \textit{$\Sigma$-periodic}). More precisely, we define 
\begin{align*}
C_\#^\infty(Z)\coloneqq\{\phi\in C^\infty(\mathbb{R}^2\times [-1, 1]) \ : \ \phi (y + e_i) = \phi(y),\, y\in Z, \, i=1,2\},
\end{align*}
and denote by $\displaystyle H^1_\#(Z)$ its completion with respect to the $\displaystyle H^1(Z)$-norm. For $*\in\{s,f\}$, we denote by $\displaystyle H^1_\#(Z_*)$ the restriction of $\displaystyle H^1_\#(Z)$-functions to $\displaystyle Z_*$. Furthermore, we define 
$$C_\#^\infty(\Omega^\pm)\coloneqq\{\phi\in C^\infty(\mathbb{R}^2\times [-H,H]) \ : \ \phi(x+l_ie_i) = \phi(x),\, x \in \Omega^\pm, \, i=1,2\},$$
where $l_i = b_i-a_i, \, i=1,2$.
Then the function space $\displaystyle H^1_\#(\Omega^\pm)$ is the completion of $\displaystyle C^\infty_\#(\Omega^\pm)$ with respect to the $\displaystyle H^1(\Omega)$-norm. 
For a Lipschitz domain $U\subset \mathbb{R}^n$ and $\omega\subset \partial U$, we define 
$$
H^1(U,\omega)\coloneqq\{u\in H^1(U) \ : u=0 \text{ on }\omega\}.
$$
In the formulation of the cell problems below, we will use the space
$$
H^1_\#(Z_f,\Gamma):= H^1_\#(Z_f) \cap H^1(Z_f,\Gamma).
$$

For a Banach space $B$ and an arbitrary open set $U\subset \mathbb{R}^m$, we write $\displaystyle L^p(U,B)$ for the usual Bochner spaces, for $p\in[1,\infty]$.

\subsection{The transmission  problem}\label{sec:transmisionProblem} 
We now introduce the mathematical model to be studied numerically in this paper. It consists of Stokes problems in the two bulk domains $\Omega^\pm$, which are coupled by \textit{effective transmission conditions} at the interface $\Sigma$, and the boundary conditions at the outer boundary of $\Omega$. The transmission conditions are given with the help of solutions to \textit{cell problems} formulated on the fluid part $Z_f$ of the periodicity cell $Z$. We focus on applications where the solid part $Z_s$ is nonempty, see e.g., \cref{fig:PeriodicityCell-sym_asym}(b) and (c). However, we are interested to understand the behavior of the solutions to the transmission problem also in comparison to the situation when no obstacle is present, that is for $Z_s=\emptyset$, like in \cref{fig:PeriodicityCell-sym_asym}(a). 

We consider the following mathematical model for the unknown velocity and pressure $(v^+, v^-)$ and $(p^+,p^-)$
\begin{subequations}\label{prob:macroscopic}
\begin{align}
- \nabla\cdot D(v^\pm) + \nabla p^\pm &= f^\pm 
    & \mbox{ in } \Omega^\pm,\label{eq:macro_a}\\[1ex]
\nabla\cdot v^\pm &= 0 
    &\mbox{ in } \Omega^\pm,\label{eq:macro_b} \\[1ex] 
    [v^+]_3 &= [v^-]_3 
    &\mbox{ on } \Sigma,\label{eq:macro_c}
    \\[1ex]
    -\llbracket (D(v) - p I)\nu\cdot\nu\rrbracket
 &= K^+ v^+ \cdot\nu^+ - K^- v^-\cdot\nu^- 
 &\mbox{ on } \Sigma,\label{eq:macro_Tnormal}\\[1ex] [ (D(v^\pm) - p^\pm I)\nu^\pm]_t
 &= -[K^\pm v^\pm]_t - M^\mp v^\mp 
 &\mbox{ on }\Sigma.\label{eq:macro_Ttangential}
\end{align}
\end{subequations}
To close the system, we impose boundary conditions on the outer boundary of $\Omega$. These may be mixed Dirichlet and Neumann boundary conditions (which also include the pure Dirichlet and pure Neumann ones)
\begin{subequations}\label{eq:bound_cond_Dir_Neumann}
\begin{align}
v^\pm &= v_D^\pm 
    &\mbox{ on }\partial_D \Omega^\pm, \label{eq:macro_D}\\[1ex] 
    -[D(v^\pm) - p^\pm I]\nu^\pm &= g^\pm 
    &\mbox{ on }\partial_N \Omega^\pm, \label{eq:macro_N}
\end{align}
\end{subequations}
or, alternatively, Dirichlet and/or Neumann boundary conditions on the top and bottom boundaries $\Gamma^\pm_\text{cap} :=\Sigma \times \{\pm H\}$, and periodic boundary conditions at the lateral boundary of $\Omega$, see also Scenarios A-D in Section~\ref{sec:simul_transm_model}.

In the transmission conditions \eqref{eq:macro_Tnormal}-\eqref{eq:macro_Ttangential}, $\llbracket \phi \rrbracket \coloneqq \phi^+ - \phi^-$ on $\Sigma$ denotes a jump across $\Sigma$, and $[\cdot]_t$ the tangential part of a vector field.
$K^\pm, \ M^\pm \in \mathbb{R}^{3\times 3}$ are the effective permeability coefficients defined for $\alpha \in \{\pm\}$ as
\begin{equation} \label{Effective_coeff}
\begin{aligned}
K^\alpha_{ij} &:= \int_{Z_f} D_y(q_i^\alpha) : D_y(q_j^\alpha) \, dy, \hspace{0.4cm}i,j \in \{1,2\},\\
K^\alpha_{i3} := K^\alpha_{3i} &:= \int_{Z_f} D_y(q_i^\alpha) : D_y(q_3) \, dy,\hspace{0.4cm}i \in \{1,2\}, \\
K^\alpha_{33} &:= \frac{1}{2}\int_{Z_f} D_y(q_3) : D_y(q_3) \, dy, \\
M^+_{ij} &:= \int_{Z_f} D_y(q_j^+) : D_y(q_i^-) \, dy, \hspace{0.4cm}i,j \in \{1,2\},\\
M^-_{ij} &:= \int_{Z_f} D_y(q_j^-) : D_y(q_i^+) \, dy,\hspace{0.4cm}i,j \in \{1,2\}, \\
M^\alpha_{3i} &:= M^\alpha_{i3} := 0, \hspace{0.4cm}i \in \{1,2,3\}.
\end{aligned}
\end{equation}
Here, the tuple $(q_i^\pm,\pi_i^\pm)\in H^1_\#(Z_f,\Gamma)^3 \times L_0^2 (Z_f)$ for $i\in \{1,2\}$ is the unique weak solution of the Stokes cell problem
\begin{subequations}\label{prob:cellprob1-4}
\begin{align}
-\nabla_y \cdot D_y(q_i^\pm) + \nabla_y \pi_i^\pm &= 0
    &\inS{Z_f,}\label{eq:cellprob1-4_a}\\[0.7ex]
\nabla_y\cdot q_i^\pm &= 0 
    &\inS{Z_f,}\label{eq:cellprob1-4_b}\\[0.7ex] q_i^\pm &=0  
    &\onS{\Gamma \cup S^\mp,} \label{eq:cellprob1-4_c}\\[0.7ex]
q_i^\pm &= e_i 
    &\onS{S^\pm,}\label{eq:cellprob1-4_d}\\[0.7ex]  q_i^\pm, \ \pi_i^\pm &\text{ is } Y\text{-periodic, }\label{eq:cellprob1-4_e}
\end{align}
\end{subequations}
and $(q_3,\pi_3)\in H^1_\#(Z_f,\Gamma)^3 \times L_0^2 (Z_f)$ is the unique weak solution of
\begin{subequations}\label{prob:cellprob5}
\begin{align}
-\nabla_y \cdot D_y(q_3) + \nabla_y \pi_3 &= 0
    &\inS{Z_f,}\label{eq:cellprob5_a}\\[0.7ex]
\nabla_y \cdot q_3 &= 0 
    &\inS{Z_f,}\label{eq:cellprob5_b}\\[0.7ex] q_3 &=0  
    &\onS{\Gamma,}\label{eq:cellprob5_c}\\[0.7ex]
q_3 &= e_3  
    &\onS{S^+ \cup S^-,}\label{eq:cellprob5_d} \\[0.7ex]  q_3, \ \pi_3 &\text{ is } Y\text{-periodic.}\label{eq:cellprob5_e}
\end{align}
\end{subequations}
\begin{remark}
The transmission model \eqref{prob:macroscopic} together with homogeneous boundary conditions \eqref{eq:bound_cond_Dir_Neumann}  was derived in \cite{GahnNeussRadu2026} under the assumption $\partial_N \Omega \cap (\partial \Sigma \times (-\eps,\eps)) = \emptyset$.  Furthermore, a time-dependent version of problem~\eqref{prob:macroscopic} was derived before in \cite{NeussRaduGahn2025} for periodic boundary conditions at the lateral boundary of $\Omega$ and homogeneous boundary conditions for the normal stress at the bottom and top boundaries, more precisely,
\begin{subequations}
\begin{align}
       -[D(v^\pm) - p^\pm I]\nu &= 0 
    &\mbox{ on } \Gamma^\pm_\text{cap}, \label{eq:macro_NPer}\\[1ex]
    v^\pm, p^\pm \text{ is } &\Sigma\text{-periodic.} \label{eq:macro_h} 
\end{align}
\end{subequations}
\end{remark}

In the following we give the weak formulation of the transmission model \eqref{prob:macroscopic} together with boundary conditions \eqref{eq:bound_cond_Dir_Neumann} and prove well-posedness of this model. This result will then be used in Section~\ref{sec:macroscopicSimulations} for the analysis of the discretized model based on finite element approximations.
\\

\noindent\textbf{Assumptions on the data:} 
\smallskip
\begin{enumerate}[label = (A\arabic*)]
\item \label{ass:rhs_f_pm} For the bulk forces we assume $f_0^\pm \in L^2(\Omega^{\pm})$.

\item \label{ass:g_vD} For the boundary data, we assume that
$$
\left\{ \begin{array}{ll} 
v_D^\pm \in  H^1(\Omega^\pm),\,  v_D^\pm \mbox{ has compact } & \\
\mbox{ support in } \overline{\Omega^\pm}\setminus \Sigma,\,\, \nabla \cdot v_D^\pm =0, \mbox{ and } & \\
g^\pm \in L^2( \partial_N\Omega^\pm) & |\partial_N \Omega| >0\\
\\
v_D^\pm \in  H^1(\Omega^\pm), \, v_D^\pm \mbox{ has compact } & \\
\mbox{ support in } \overline{\Omega^\pm}\setminus \Sigma,\, \mbox{ and } \ \nabla \cdot v_D^\pm =0  & \partial_N \Omega = \emptyset.
\end{array}\right.
$$
\end{enumerate}
Let us define the function space
$$
X= H^+ \times H^-
$$
where
$$
    H^\pm = H^1(\Omega^\pm, \partial_D \Omega^\pm)^3.
$$
We will identify elements $v=(v^+,v^-) \in X$
with functions on $\Omega$ which are discontinuous
across $\Sigma$ and write $v^\pm = v|_{\Omega^\pm}$
for the restriction onto the subdomains $\Omega^+$ and $\Omega^-$.
Note that according to \eqref{Def_boundary_parts}, we have $\partial_D\Omega^\pm \subseteq (\partial \Omega^\pm \setminus \Sigma)$, whereby equality holds for $\partial_N\Omega^\pm = \emptyset$.
Further, we consider the subspace
\begin{equation*}
H = \left\{ \phi\in X: [\phi^+]_3 = [\phi^-]_3 
\,\mbox{on } \Sigma \right\}
= X \cap H^{\operatorname{div}}(\Omega)
\end{equation*} 
of functions with continuous normal component
on $\Sigma$ (and thus well defined weak divergence)
and the subspace
\begin{equation*}
V = \left\{ \phi \in H: \, \nabla \cdot \phi = 0 \mbox{ in } \Omega^+ \cup \Omega^- \right\}
    = \left\{ \phi \in H: \, \nabla \cdot \phi = 0 \mbox{ in } \Omega \right\}.
\end{equation*} 
of divergence free functions.
Finally, let
\begin{align}
\label{eq:pressure_space}
Q= \left\{ \begin{array}{ll} 
L_0^2(\Omega), & \partial_N \Omega = \emptyset\\
L^2(\Omega), & |\partial_N \Omega| >0.
\end{array}\right.
\end{align}
We call the tuple $(v,p)$ a weak solution of the transmission problem~\eqref{prob:macroscopic} together with boundary conditions \eqref{eq:bound_cond_Dir_Neumann}, if $(v-v_D,p) \in V \times Q$, and for all $\phi\in H$ it holds  that
\begin{eqnarray}
\label{eq:macro_weak_form}
&&\sum_{\pm} \left\{ \int_{\Omega^{\pm}} D(v^{\pm}) : D(\phi^{\pm}) dx -  \int_{\Omega^{\pm}} p^{\pm} \nabla \cdot \phi^{\pm} dx   \right\}  \nonumber
\\
&&+ \sum_{\pm} \int_{\Sigma} K^{\pm} v^{\pm} \cdot \phi^{\pm} d\sigma + \int_{\Sigma} M^{-}v^- \cdot \phi^+ + M^{+}v^+ \cdot \phi^- d\sigma \label{eq:weak_formul_eff_model} \\
&& = \sum_{\pm} \int_{\Omega^{\pm}} f^{\pm} \cdot \phi^{\pm} dx + \int_{\partial_N\Omega} g^\pm \cdot\phi^\pm d\sigma.\nonumber
\end{eqnarray}

While existence and uniqueness of solutions to~\eqref{eq:macro_weak_form}
could be shown using a limiting argument similar to \cite{NeussRaduGahn2025,GahnNeussRadu2026},
we will prove it manually here using classical linear existence theory.
To this end we define the bilinear and linear forms
\begin{align*}
    a^\pm &: H^\pm \times H^\pm \to \mathbb{R},&
    a^\pm(v^\pm, \phi^\pm)
        &= \int_{\Omega^\pm} D(v^\pm) : D(\phi^\pm) dx, \\
    a^\Sigma &: H \times H \to \mathbb{R},&
    a^\Sigma(v,\phi)
        &= \int_\Sigma \mathcal{K} v \cdot \phi d\sigma, \\
    a &: H \times H \to \mathbb{R},&
    a(v,\phi)
        &= a^+(v^+,\phi^+) + a^-(v^-,\phi^-) + a^\Sigma(v,\phi),\\
    b^\pm &: H^\pm \times L^2(\Omega^\pm) \to \mathbb{R},&
    b^\pm(v^\pm,p^\pm)
        &= -\int_{\Omega^\pm} p^\pm \nabla \cdot v^\pm dx,\\
    b &:H \times Q \to \mathbb{R}, &
    b(v,p)
        &= b^+(v^+,p^+) + b^-(v^-,p^-)
        = - \int_\Omega p \nabla \cdot v dx,\\
    \ell &: H \to \mathbb{R}, &
    \ell(\phi)
        &= \sum_{\pm} \int_{\Omega^{\pm}} f^{\pm} \cdot \phi^{\pm} dx + \int_{\partial_N\Omega} g^\pm \cdot\phi^\pm d\sigma
\end{align*}
where $\mathcal{K}v \cdot \phi$ has to be understood
as symbolic matrix-vector-vector product of the matrix
$$
    \mathcal{K} = \begin{pmatrix} K^+ & M^-\\M^+&K^-\end{pmatrix}
$$
and the vectors $v=(v^+,v^-)$ and $\phi=(\phi^+,\phi^-)$.
Using this notation we can rewrite~\eqref{eq:macro_weak_form}
as the following saddle point problem:
Find $(v-v_D, p) \in H \times Q$ such that
\begin{subequations}
\label{eq:weak_problem}
\begin{alignat}{3}
    &a(v,\phi)&{} + b(\phi,p) &= \ell(\phi)
        \qquad&\forall \phi \in H,\\
    &b(v,q)& &= 0
        & \forall q \in Q.
\end{alignat}
\end{subequations}

To show existence we make use of the following
elementary lemma.

\begin{lemma}
    \label{eq:infsup_subspaces}
    Let $X$ and $Y$ be inner product spaces and
    $b:X\times Y \to \mathbb{R}$ a bilinear form.
    Assume that $X_i \subset X$ and $Y_i \subset Y$, $i=1,2$ are subspaces such that $Y=Y_1 \oplus Y_2$ and the $X_i$ are biorthogonal
    to the $Y_i$ with respect to $b(\cdot,\cdot)$,
    i.e. for $i\neq j$ we have
    $b(v_i,p_j) = 0$ for $v_i \in X_i$ and $p_j \in Y_j$.
    Furthermore assume that an inf-sup-condition holds true
    on each pair $X_i \times Y_i$, i.e.,
    \begin{align*}
        \exists \beta_i >0 \,  \forall p_i \in Y_i: \qquad
        \sup_{\substack{v_i \in X_i,v_i \neq 0}}
        \frac{b(v_i, p_i)}{\|v_i\|} \geq \beta_i \|p_i\|.
    \end{align*}
    Then the inf-sup-condition also holds true on the pair
    $X \times Y$ with the constant
    $\beta = \frac{1}{2}\min\{\beta_1,\beta_2\}$.
\end{lemma}
\begin{proof}
    Let $p \in Y$. Then we can write $p = p_1 + p_2$
    with $p_i \in Y_i$.
    Using the inf-sup-conditions for $(X_i,Y_i)$ we get
    \begin{align*}
        \sup_{\substack{v \in X,v \neq 0}}
            \frac{b(v, p)}{\|v\|}
        \geq \sup_{\substack{v_i \in X_i,v_i \neq 0}}
            \frac{b(v_i, p)}{\|v_i\|}
        = \sup_{\substack{v_i \in X_i,v_i \neq 0}}
            \frac{b(v_i, p_i)}{\|v_i\|}
        \geq \beta_i \|p_i\|
        \geq 2\beta \|p_i\|.
    \end{align*}
    Taking the maximum over $i$ and using
    $\|p\| \leq 2\max\{\|p_1\|, \|p_2\|\}$
    yields the assertion. 
\end{proof}

\begin{proposition}
\label{prop:weak_problem_existence}
Let $Z_s \neq \emptyset$. Then, there exists a unique weak solution $(v,p)$ to the transmission problem~\eqref{prob:macroscopic} together with boundary conditions \eqref{eq:bound_cond_Dir_Neumann}. 
\end{proposition}

\begin{proof}
We first show coercivity of $a(\cdot,\cdot)$ on $H$.
To this end we note that, following~\cite[Lemma 5]{NeussRaduGahn2025},
there exists a constant $c_0>0$ such that for every
$\xi^{\pm} \in \R^3$ with $\xi_3^+ = \xi_3^-$ it holds that \begin{align}\label{ineq:coercivity_macro_fluid}
\mathcal{K} \xi \cdot \xi =
M^+\xi^+ \cdot \xi^- + M^- \xi^- \cdot \xi^+ + \sum_{\pm} K^{\pm} \xi^{\pm} \cdot \xi^{\pm} \geq c_0 \sum_{\pm} |\xi^{\pm}|^2,
\end{align}
i.e., the symmetric matrix $\mathcal{K}$ is positive definite
on the corresponding 5-dimensional subspace.
As a consequence, we get
$$
    a(v,v)
        \geq \sum_\pm \int_{\Omega^\pm} D(v^\pm)^2 dx
            + c_0 \int_\Sigma |v^\pm|^2 d\sigma
        \geq C \sum_\pm \|v^\pm\|^2_{H^1(\Omega^\pm)} = C \|v\|_H^2
$$
with $C>0$, where the second inequality follows from 
Korn's second inequality (see, e.g.,
\cite[Proposition 2] {Graeser2015Poincare} and \cite{MosMja71}.) 

To show the inf-sup (or Ladyzhenskaya–Babuška–Brezzi)
condition for the bilinear form
$b(\cdot,\cdot)$ on $H \times Q$
it will be helpful to consider the closed subspaces
\begin{align*}
    H^\pm_0
        &= \{v^\pm\in H^\pm\,:\, v^\pm = 0 \text{ on }\Sigma\}
        = H^1(\Omega^\pm, \partial_D\Omega^\pm\cup\Sigma)^3
        \subset H^\pm,\\   
    Q^\pm  
        &= \begin{cases}
            L^2(\Omega^\pm) &\text{ if }|\partial_N \Omega^\pm|>0,\\
            L^2_0(\Omega^\pm) &\text{ else}
            \end{cases}
        \subset L^2(\Omega^\pm).
\end{align*}
It is well-known, that the subdomain bilinear forms
$b^\pm(\cdot,\cdot)$ satisfy the inf-sup
conditions on $H^\pm_0 \times Q^\pm$, respectively,
i.e., 
$$
    \exists \beta^\pm>0\,
    \forall p^\pm \in Q^\pm:
    \qquad
    \sup_{\substack{v^\pm \in H^\pm_0\\v^\pm \neq 0}}
    \frac{b^\pm(v^\pm, p^\pm)}{\|v\|_{H^1(\Omega^\pm)}}
    \geq \beta^\pm \|p^\pm\|_{L^2(\Omega^\pm)}.
$$
Extending functions from $H^\pm_0$ and $Q^\pm$ by zero
to the other subdomain, we can view
these spaces as subspaces of $(H^+_0 \times H^-_0)$
and $(Q^+ \times Q^-)$, respectively.
Using Lemma~\ref{eq:infsup_subspaces} with $i\in \{+,-\}$
then yields an inf-sup condition on
$(H^+_0 \times H^-_0) \times (Q^+ \times Q^-)$
$$
    \exists \beta>0\,
    \forall p \in Q^+\times Q^-:
    \qquad
        \sup_{\substack{v \in H^+_0 \times H^-_0\\v \neq 0}}
        \frac{b(v, p)}{\|v\|_X}
    \geq \beta \|p\|_{L^2(\Omega)}.
$$
If both
$|\partial_N\Omega^+| >0$ and  $|\partial_N\Omega^-| >0$ hold true,
then this already implies the desired
condition on $H \times Q$, because $H\supset H^+_0 \times H^-_0$
and $Q=Q^+ \times Q^-$.

For the remaining cases where $|\partial_N\Omega^\pm|=0$
for at least one subdomain we will again use
Lemma~\ref{eq:infsup_subspaces}, now with $Y_1 = Q^+ \times Q^-$,
$X_1 = H_0^+ \times H_0^-$,
$Y_2=Y_1^\perp = \operatorname{span} \{\hat{p}\}$,
and $X_2=\operatorname{span} \{\hat{v}\}$
for suitable $\hat{p} = (\hat{p}^+, \hat{p}^-)\in Q$
and $\hat{v} = (\hat{v}^+, \hat{v}^-)\in H$.
In all cases bi-orthogonality is straight forward and we will
only show $b(\hat{v},\hat{p}) \neq 0$ which is equivalent
to the inf-sup-condition on $X_2 \times Y_2$.

If $|\partial_N\Omega^+| = |\partial_N\Omega^-| = 0$
we use $\hat{p}^+ \equiv 1$ and $\hat{p}^- \equiv -1$.
Since $\operatorname{div} : H^+ \to L^2(\Omega^+)$
is known to be surjective, there is a $\hat{v}^+ \in H^+$
with $\operatorname{div} \hat{v}^+ = 1$. Furthermore we can construct
a $\hat{v}^- \in H^-$ such that $\operatorname{div} \hat{v}^- = -1$
and $[\hat{v}^+]_3 = [\hat{v}^-]_3$ on $\Sigma$.
Then we easily compute  $b(\hat{v}, \hat{p}) = - |\Omega| \neq 0$.

Now let, w.l.o.g., $|\partial_N\Omega^+| = 0$ but
$|\partial_N\Omega^-| > 0$.
Then define $\hat{p}$ and $\hat{v}$ as above with the
exception that we now set $\hat{p}^- \equiv 0$ and
use a $\hat{v}^-$ with $\operatorname{div} \hat{v}^- = 0$
which is possible because $H^-_0 \neq H_0^1(\Omega^-)$.
Then we have $b(\hat{v}, \hat{p}) = - |\Omega^+| \neq 0$.
\end{proof}

\begin{remark}
  We could alternatively derive the inf-sup condition on
  $H \times Q$ in a more direct way from the classical inf-sup condition on
  the subspace
  $H^1(\Omega, \partial_D\Omega) \times Q$
  where tangential velocities are continuous, too.
  However, we opted for the given proof
  because the same construction will be
  used to show inf-sup stability for
  a finite element discretization later on.
\end{remark}

\begin{remark}\label{rem:well-posedness}
For the case $Z_s = \emptyset$ inequality \eqref{ineq:coercivity_macro_fluid}  does not hold. To see this, we use the effective coefficients computed in \eqref{eq:coeff_empty} and take $\xi^+ =\xi^- = \xi_0$, where $\xi_0 \in \R^3$ is a constant vector, to obtain
$$
M^+\xi^+ \cdot \xi^- + M^- \xi^- \cdot \xi^+ + \sum_{\pm} K^{\pm} \xi^{\pm} \cdot \xi^{\pm} = 0.
$$
Thus, the result of the proposition holds for boundary conditions \eqref{eq:bound_cond_Dir_Neumann} except for the pure Neumann ones. In the latter case, a compatibility condition for the data has to be required to ensure that the problem is well-posed.
\end{remark}

\begin{remark}
    The result of Proposition~\ref{prop:weak_problem_existence}
    can be directly generalized to the case of periodic
    boundary conditions with periodicity in either
    $x_1$- or $x_2$-direction or in both of these directions.
    In this case the pressure space $Q$ is still given
    as in \eqref{eq:pressure_space} depending on
    $\partial_N \Omega$.
\end{remark}

\section{Numerical computation of cell solutions and effective coefficients} \label{sec:numerical_cellSolCoeff}

The aim of this section is to compute numerically the effective coefficients $K^\pm, \ M^\pm\in\mathbb{R}^{3\times 3}$ entering the interface conditions \eqref{eq:macro_Tnormal}-\eqref{eq:macro_Ttangential} in the transmission problem  \eqref{prob:macroscopic}. 

Let us first mention that for the case when $Z_s=\emptyset$,  the solutions to the cell problems \eqref{prob:cellprob1-4} and \eqref{prob:cellprob5} and the effective coefficients $K^\pm$ and $M^\pm$ can be calculated explicitly. Namely, the cell solutions are given by
\begin{subequations}
\begin{align}\label{eq:cell_soll_empty}
q^\pm_i(y) &= \left(\pm\frac{1}{2}y_3 + \frac{1}{2}\right) e_i, &\pi_i^\pm(y) &= 0, \quad i \in \{1,2\},\\
q_3(y) &= e_3, &\pi_3(y) &= 0.
\end{align}
\end{subequations}
These yield the effective coefficients 
\begin{align}\label{eq:coeff_empty}
K^\pm = \text{diag}\left(\frac{1}{4},\frac{1}{4},0 \right) \quad \text{ and } \quad M^\pm = \text{diag}\left(-\frac{1}{4},-\frac{1}{4},0 \right).
\end{align}

For the situation when $Z_s\neq \emptyset$, it turns out, that the main difficulty in the numerical calculation of the coefficients lies in the computation of the solutions  $(q_i^\pm,\pi_i^\pm),\, i\in \{1,2\}$, and $(q_3,\pi_3)$ of the Stokes cell problems \eqref{prob:cellprob1-4} and \eqref{prob:cellprob5}, respectively. This is due to the fact that the computational domain $Z_f$ (the fluid part of the periodicity cell) is three-dimensional and complex, leading to challenges in the construction of finite element meshes and to high computational cost. A lower-dimensional reduction is not feasible, as the fluid region of the reference cell must remain connected. Moreover, we need to deal with periodic boundary conditions, which cause considerable effort to construct a suitable mesh and to effectively enforce the periodicity to both velocity and pressure.

We perform the numerical computations for two types of microstructure, a symmetric one, where the solid phase $Z_s$ consists of the union of cylinders with rotational axes parallel to the three coordinate axes, and an asymmetric one, which contains an inclined segment, see \cref{fig:PeriodicityCell-sym_asym}(b) and (c). For the symmetric cell geometry, we denote by $r$ the common radius of the cylinders, and by $h$ the height of the cylinder whose rotational axis is parallel to the $y_3$-axis.
\begin{figure}[H]
    \centering
    \subfloat[]{
        \includegraphics[height=3.2cm]{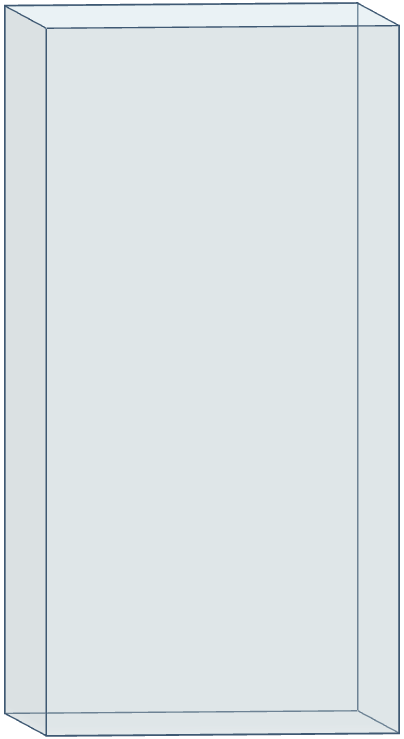}}
        \hspace{2cm}
    \subfloat[]{
        \includegraphics[height=3.2cm]{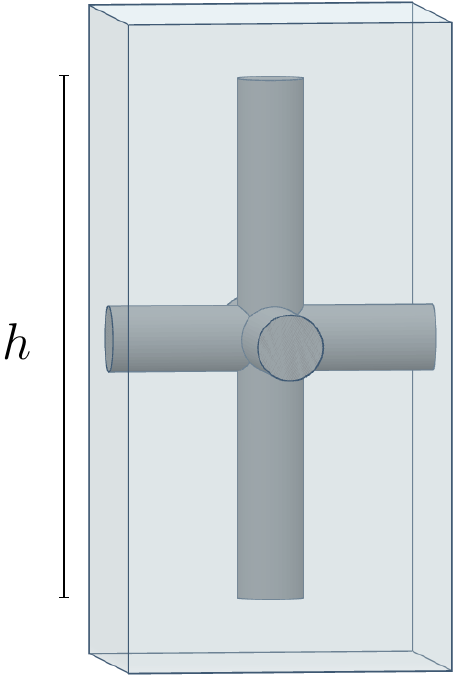}}
        \hspace{2cm}
    \subfloat[]{
        \includegraphics[height=3.2cm]{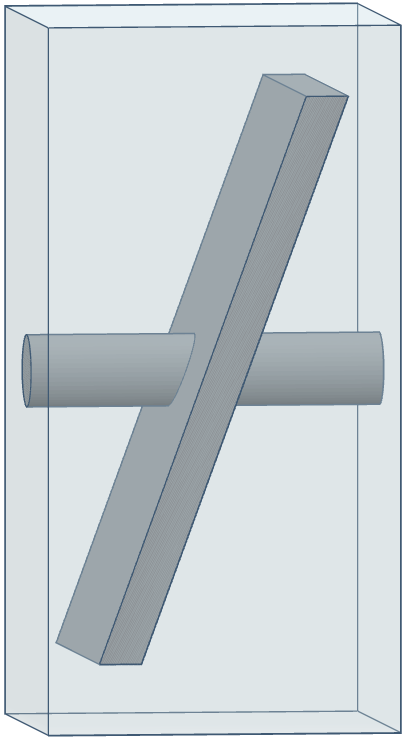}}
\caption{Examples of periodicity cells $Z$: (a) without a solid phase, (b) with symmetric geometry involving the cylindrical cross, and (c) with asymmetric geometry containing an inclined segment.} 
\label{fig:PeriodicityCell-sym_asym}
\end{figure}%

\subsection{Mesh generation and numerical environment}\label{subsec:meshGen}
The construction of the computational domain $Z_f$, as well as mesh generation and refinement are done in \texttt{gmsh}-4.13.1, see \cite{gmsh}. We use set operations in \texttt{gmsh} to construct the complex three-dimensional geometries from basic geometric entities. E.g., the cross-shaped solid phase $Z_s$ of the symmetric periodicity cell is obtained via Boolean union of cylinders (\cref{fig:PeriodicityCell-sym_asym}) or cuboids (\Cref{fig:cuboidMeshes}). The fluid part $Z_f$ is the Boolean difference between the box $Z=[0,1]^2\times[-1,1]$ and the solid phase $Z_s$. To obtain finer meshes, a first attempt was to use the refine-by-splitting strategy of \texttt{gmsh}. However, this proved unfeasible for the cylindrical, cross-shaped domain; \texttt{gmsh} issued an error message stating that, for very fine meshes, it was not possible to maintain the periodicity on opposing faces. Thus, for mesh refinement, we manipulate the \texttt{gmsh} parameter \textit{target mesh size} (\texttt{lc}) that allows for local control of element sizes. Hence, a new grid is generated for each specified value of \texttt{lc} that meets the specified fineness conditions, see \cref{fig:cylinderMeshes}.

To impose periodic boundary conditions on the perforated lateral walls, it is essential that each node on one periodic boundary has a corresponding node on the opposite periodic boundary. This way, the degrees of freedom associated with these nodes can be properly paired. In \texttt{gmsh}, the periodic mesh generation on opposite surfaces is created by translating the mesh of one surface on the opposite one, using the \texttt{Periodic Surface} functionality as follows.
\begin{lstlisting}[language=C++]
Periodic Surface {74} = {72} Translate {1, 0, 0}; 
Periodic Surface {75} = {73} Translate {0, 1, 0}; 
\end{lstlisting}
Here, the numbers $72,74$ and $73,75$ are physical surface tags used in \texttt{gmsh} to identify the opposite surfaces on which we impose periodicity in the $y_1,y_2$ directions.

The algorithm for the computation of the Stokes problems and of the effective coefficients is implemented using the \textit{Distributed and Unified Numerical Environment (DUNE)} \cite{DUNE:2021,DUNE1,EngwerEtAl2025}, which is a modular \texttt{C++} library for grid based methods for partial differential equations.
The Stokes cell problems are discretized using lowest-order Taylor--Hood finite elements via the discretization module dune-fufem, see \cite{fufem}. In particular, we employ the function \texttt{computePeriodicConstraints()} to enforce periodicity on a pair of boundary meshes defined on two surfaces named primary boundary and secondary boundary. The function will first search matching pairs of boundary faces from the primary and secondary boundaries. A pair of faces is considered of match if the center of the secondary boundary face mapped to the primary boundary coincides (up to a tolerance) with the center of the primary boundary face.
For any matching pair, the secondary basis DOFs associated to the secondary boundary face are then interpolated from the primary basis DOFs on the primary boundary face. The secondary boundary face DOFs are marked as constrained, and the interpolation weights are used as constraint weights. Thus, in order to use this function, the importance of the \texttt{Periodic Surface} functionality in \texttt{gmsh} becomes clear.
\par
To solve the arising large linear systems, we use the MINRES solver in combination with a block-diagonal preconditioner. The velocity block uses an AMG approximation of the inverse of the velocity stiffness matrix, and for the inverse Schur complement, an AMG approximation of the inverse pressure mass matrix $M_p$ is deployed. We use Dirichlet boundary conditions for the pressure mass matrix on the periodic boundaries, i.e., we modify rows and columns of $M_p$ as if the periodicity DOFs were associated with Dirichlet boundary conditions. However, we emphasize that this is a purely algebraic manipulation of $M_p$ and no Dirichlet boundary conditions are imposed on the pressure. In our tests in the case of symmetry for all considered values of $h$ and $r=0.1$, this modification leads to a reduction in the required MINRES iterations of $14\% $ with similar time per iteration for all cell problems. In general, we observe an increase in MINRES iterations as values of height $h$ and radius $r$ increase.
\begin{figure}[H]
    \centering
    \subfloat[]{
        \includegraphics[width=0.305\linewidth]{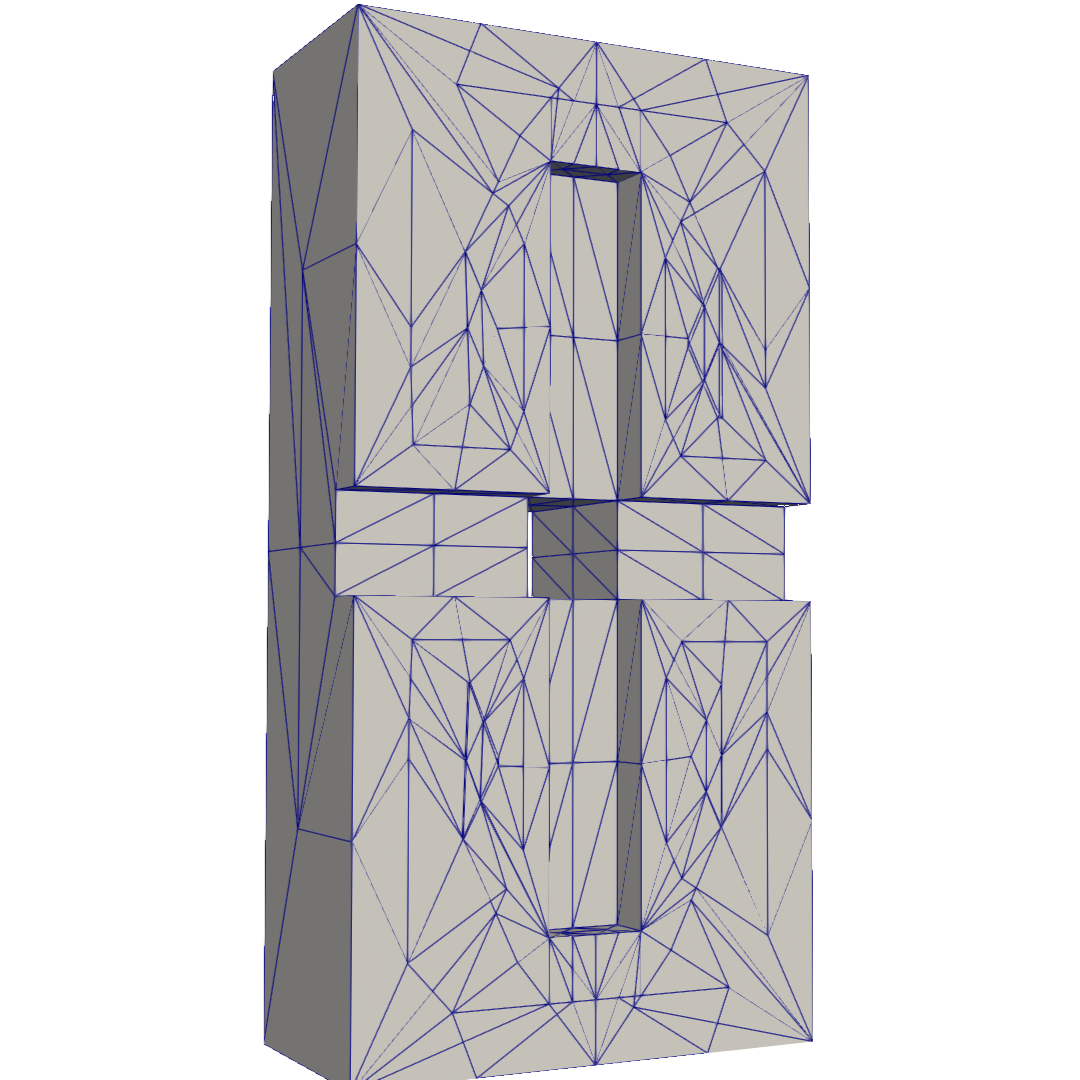}}\hfill
    \subfloat[]{
        \includegraphics[width=0.305\linewidth]{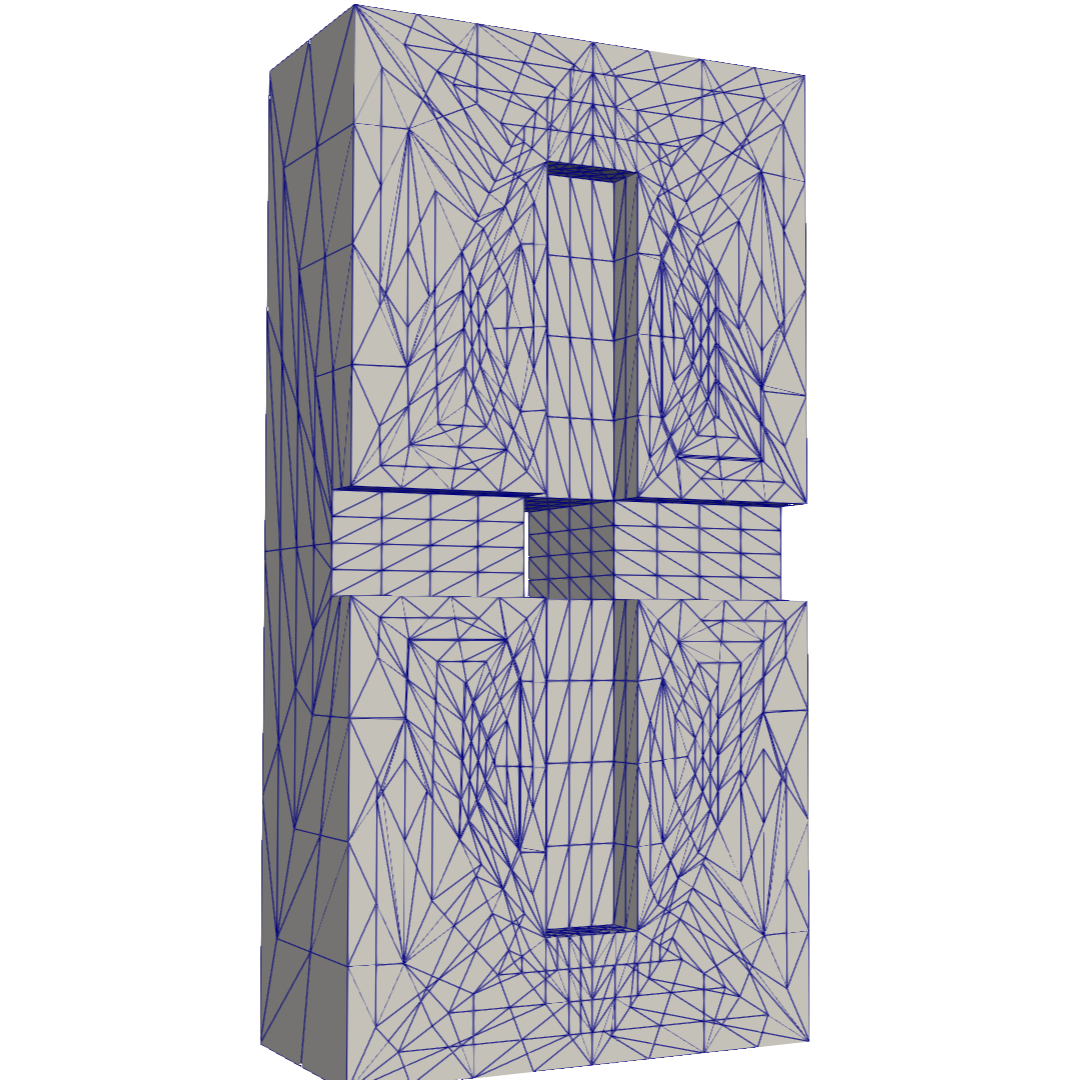}}\hfill
    \subfloat[]{
        \includegraphics[width=0.305\linewidth]{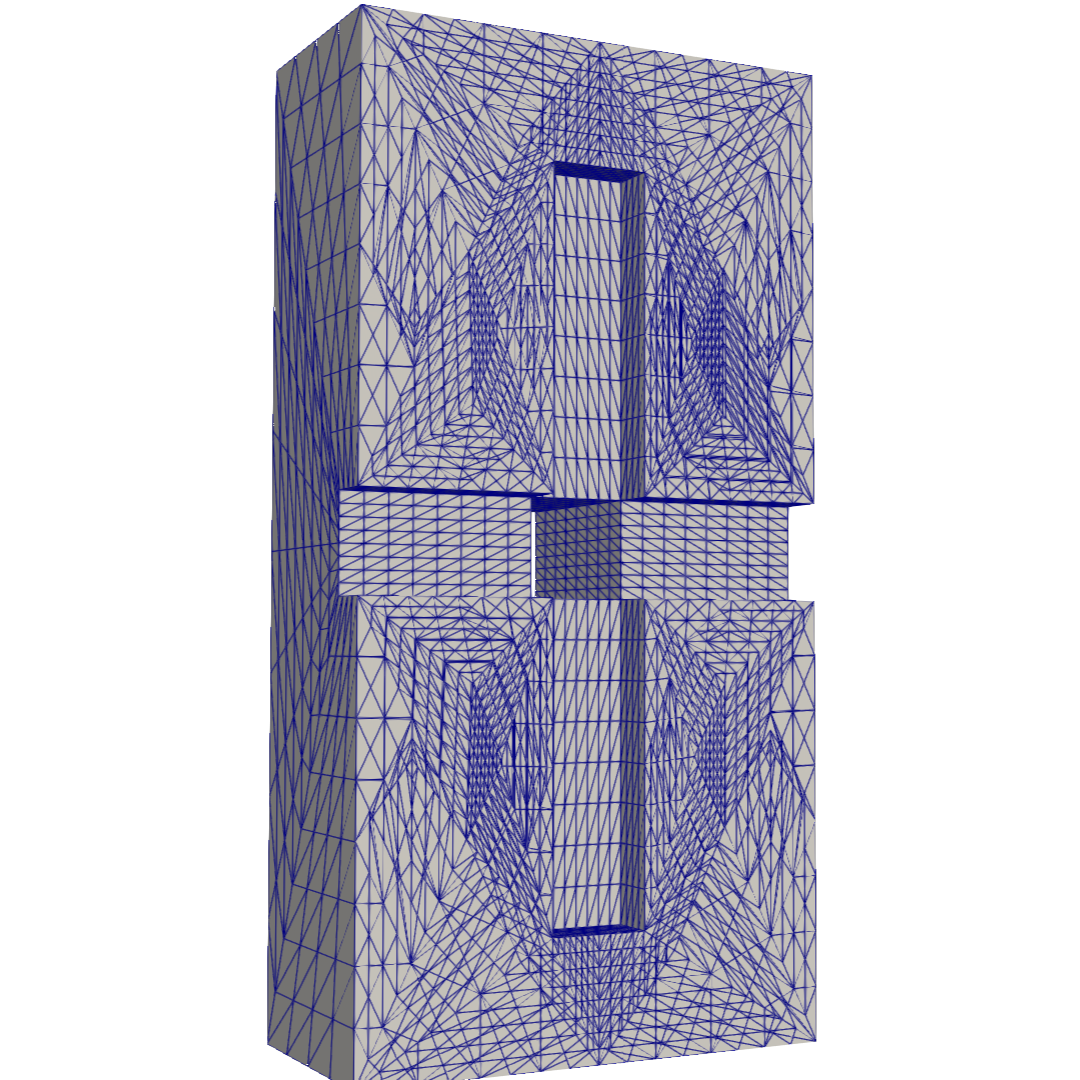}}
\caption{
Sectional views of successive mesh refinements levels $k$: (a) $k=0$, (b) $k=1$, (c) $k=2$.} 
\label{fig:cuboidMeshes}
\end{figure}%

\subsection{Numerical convergence study}\label{subsec:numStud}
In this section, we perform a mesh convergence study for selected effective coefficients. As we mentioned in the paragraph before, in \texttt{gmsh} it is not possible to generate uniformly refined meshes for the cylindrical cross-shaped domain using a refine-by-splitting strategy. Therefore, to validate our implementation, we consider a geometry including a cross-shaped domain built from cuboids instead of cylinders. In this case, \texttt{gmsh}'s refine-by-splitting allows us to obtain successively refined meshes, see \cref{fig:cuboidMeshes}, and to estimate the orders of convergence. 
In the following, let $V_k$ be the volume of the largest cell of the three-dimensional mesh at the refinement level $k\in \mathbb{N}_0$, and denote the mesh size at this refinement level by $\displaystyle h_k\coloneqq\sqrt[3]{V_k}$. As a readout, we consider the coefficients $K_{11}^+$ and $K_{33}^+$, see Problem \eqref{prob:macroscopic}. Their approximations at refinement level $k$ are denoted by $K^+_{11,k}$ and $K^+_{33,k}$, respectively.
The error at the $k$-th level is defined by $e_{k}^* \coloneqq | *_{,k-1} - *_{,k}|$ for $*\in\{K_{11}^+,K_{33}^+\}$ and $k\geq 1$. 
Taking into account that we are using lowest order Taylor--Hood finite elements and that $K_{11}^+$ and $K_{33}^+$ are defined in terms of the velocity gradient, we can expect at most quadratic convergence. A linear regression in \cref{fig:error_cuboid} confirms that we indeed obtain a global convergence order of $\alpha \approx 2$.
\begin{figure}[H]
    \centering
        \includegraphics[width=0.43\linewidth]{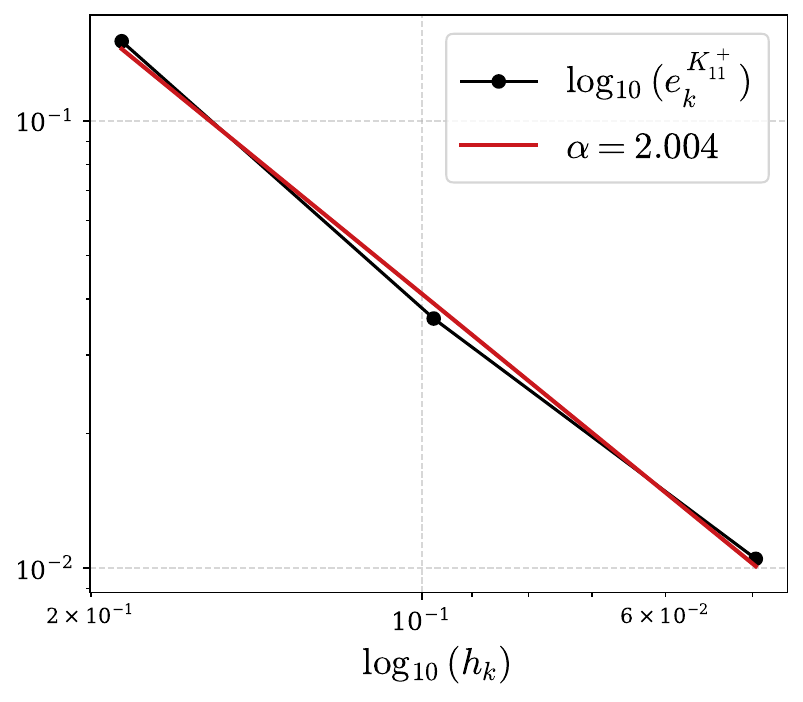}\hfill
        \includegraphics[width=0.43\linewidth]{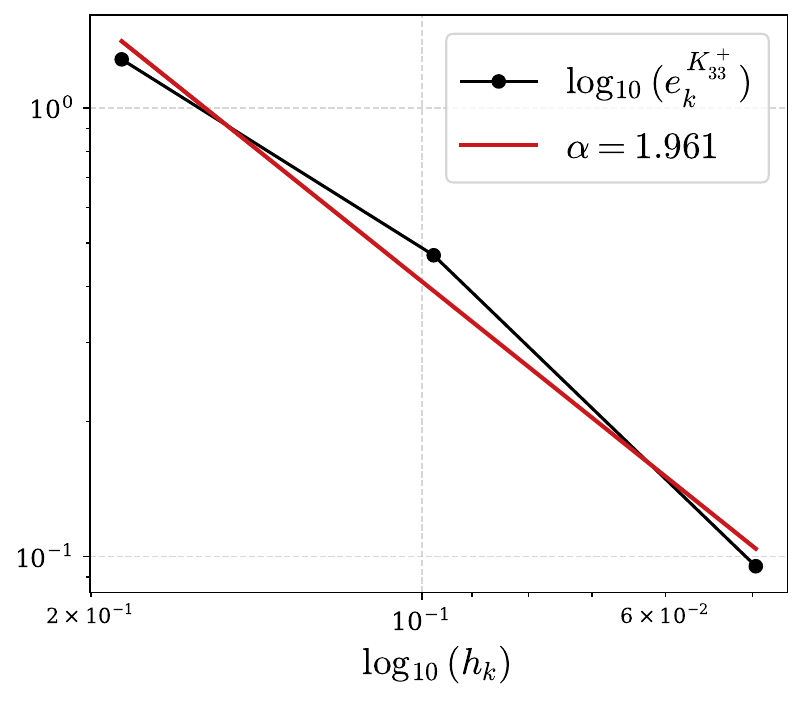}
    \caption{Plot of the errors against decreasing values of $h_k$, for the coefficient $K_{11,k}^+$ (left) and coefficient $K_{33,k}^+$ (right), shown in log–log scales.}
    \label{fig:error_cuboid}
\end{figure}
\subsection{Computation of cell solutions. Sensitivity study with respect to geometric parameters}\label{subsec:cellsolutions}
In this section, we compute the solutions $q_i^\pm, \, i\in \{1,2\}$, and $q_3$ of the Stokes cell problems \eqref{prob:cellprob1-4} and \eqref{prob:cellprob5}, respectively, for cell geometries given in \cref{fig:PeriodicityCell-sym_asym}(b) and (c). For the case (b), when the microscopic geometry is generated by the cylindrical cross, we analyze their sensitivity with respect to geometric parameters. 

For the simulation of the cell solutions, the mesh size is set to $h_k \approx 0.05$. The corresponding mesh for $h=1.5$ and $r = 0.1$ is illustrated in \cref{fig:cylinderMeshes}. We use ParaView, see \cite{Paraview}, to visualize the simulation results. 

\begin{figure}[H]
\begin{minipage}[t]{0.55\textwidth}
    \vspace{0pt}    
    \centering
    \includegraphics[width=0.605\linewidth]{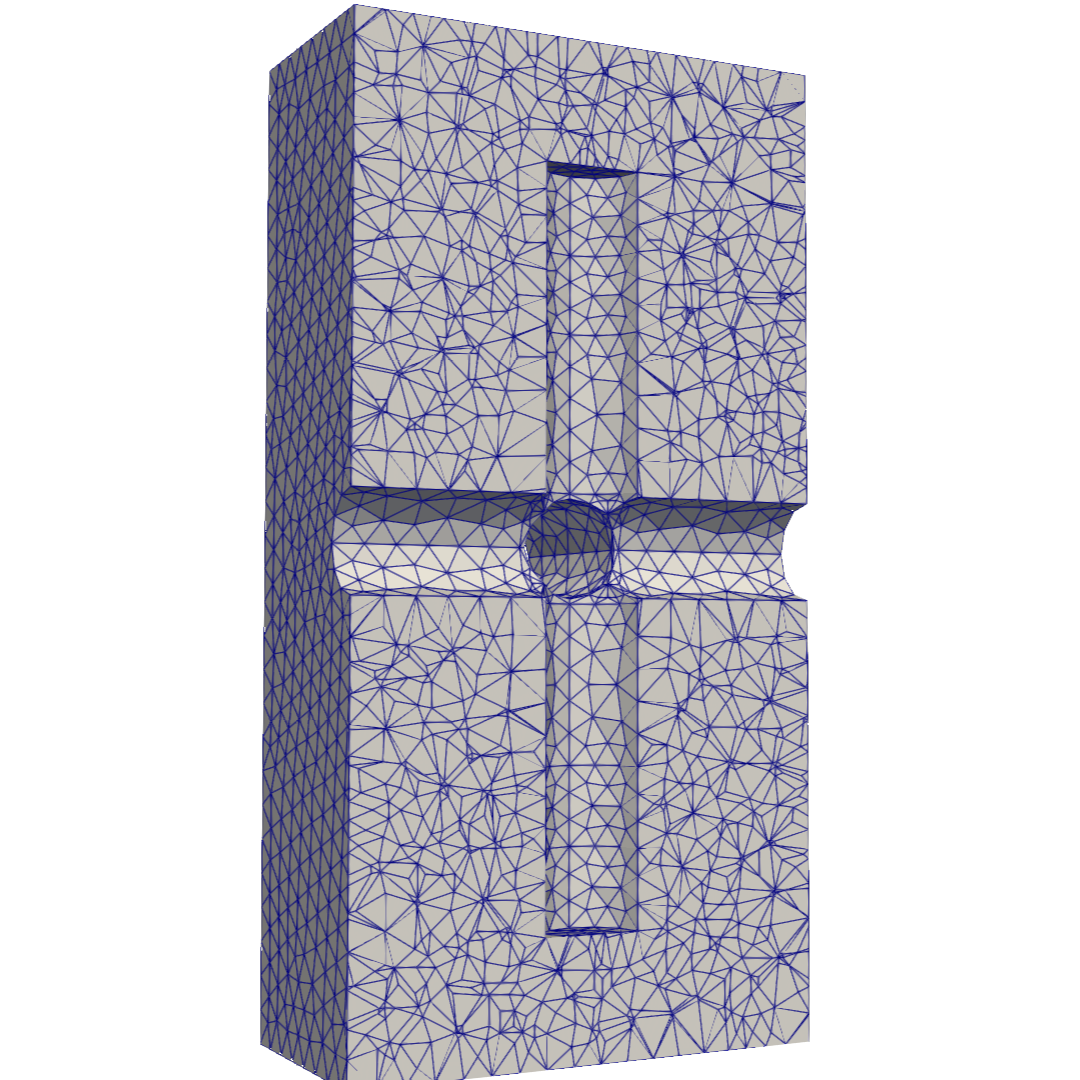}
\end{minipage}\hfill%
\begin{minipage}[t]{0.4\textwidth}
    \vspace{0pt}
    \caption{Mesh of size $h_k \approx 0.05$ for the fluid subdomain $Z_f$ of the standard cell $Z$, which was generated by a cylindrical cross with $h=1.5$ and $r = 0.1$.}
    \label{fig:cylinderMeshes}
\end{minipage}
\end{figure}
\begin{figure}[H]
    \begin{minipage}[t]{0.8\textwidth}
    \vspace{0pt}
    \centering
       \subfloat[]{
        \includegraphics[width=0.305\linewidth]{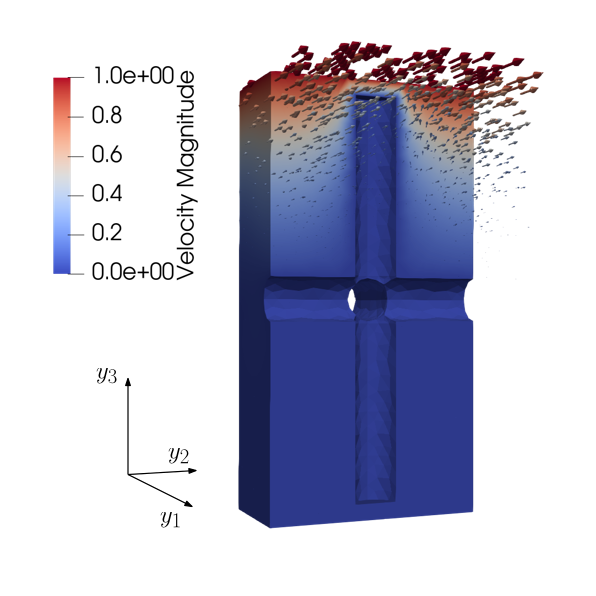}}
    \subfloat[]{
        \includegraphics[width=0.305\linewidth]{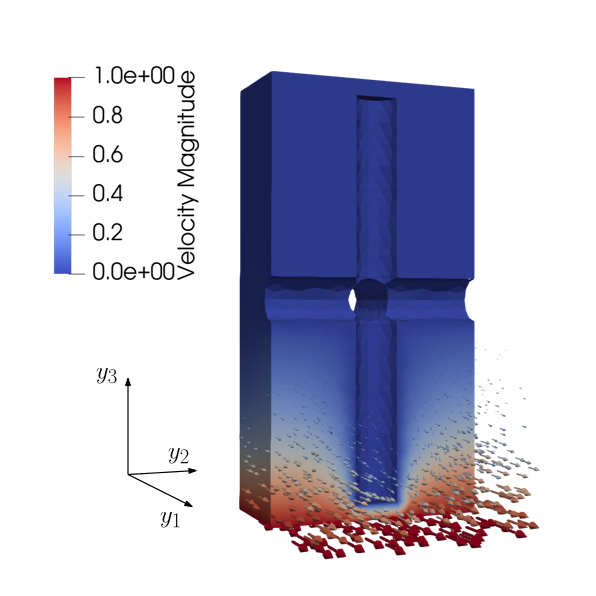}}
    \subfloat[]{
        \includegraphics[width=0.305\linewidth]{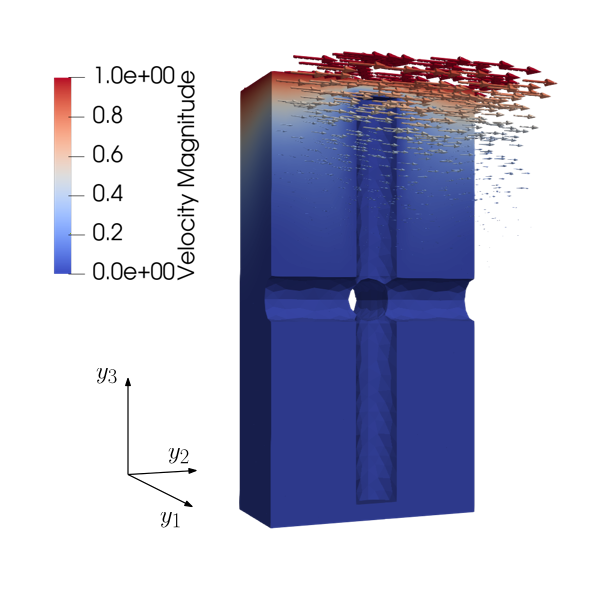}}\\
    \subfloat[]{
        \includegraphics[width=0.305\linewidth]{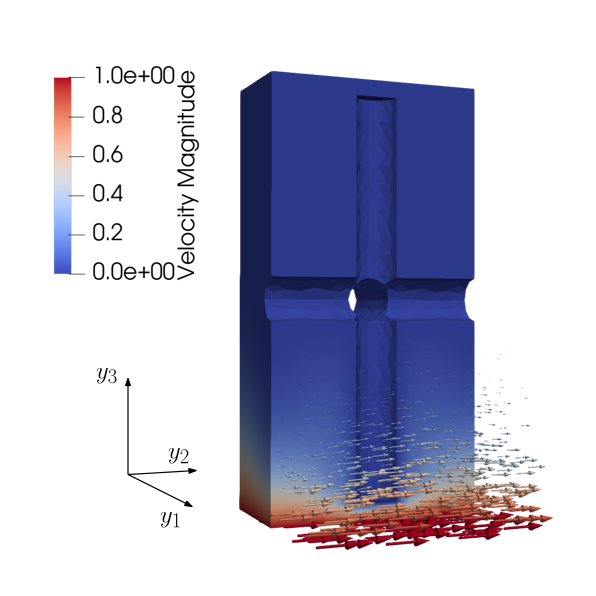}}
    \subfloat[]{
        \includegraphics[width=0.305\linewidth]{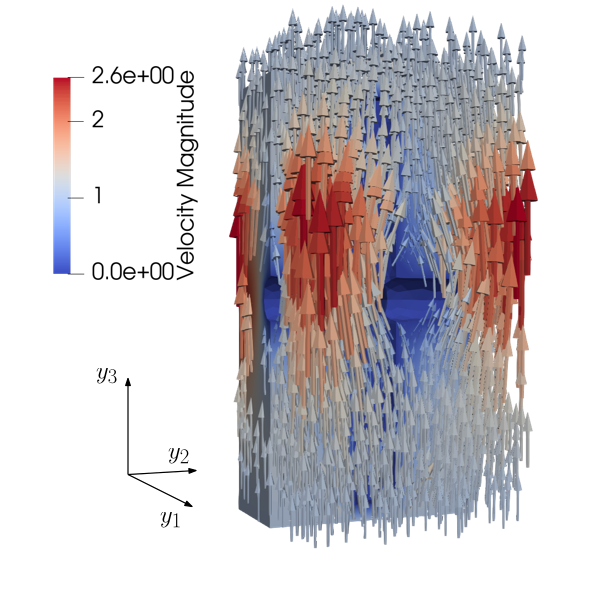}}
\end{minipage}\hfill%
\begin{minipage}[t]{0.2\textwidth}
    \vspace{10pt} 
    \caption{Sectional views of solutions $q_i^\pm, \,i\in \{1,2\}$ in (a-d) and $q_3$ in (e) to Stokes cell problems \eqref{prob:cellprob1-4} and \eqref{prob:cellprob5}, respectively, for $h = 1.85$ and $r = 0.1$, in case of the symmetric cell geometry.}
    \label{fig:5cellProblems}
\end{minipage}
\end{figure}
In \cref{fig:5cellProblems}, cell solutions computed for $h=1.85$ and $r = 0.1$ are provided. It is clearly visible that the periodic and Dirichlet boundary conditions are enforced correctly. In subfigures (a-d), the velocities $q_i^\pm$ for $i\in \{1,2\}$ develop sharp gradients in regions where different values for the Dirichlet boundary conditions are imposed on nearby boundaries close to the top or bottom of the cell. For the velocity $q_3$ in subfigure (e), the higher gradients occur in regions where the given flux is forced through the narrow lateral regions that result from the cross geometry.

As previously mentioned, we are interested in observing the sensitivity of the cell solutions to variations in the geometry of the reference cell. We investigate this for the case of the symmetric cell geometry generated by the cylindrical cross. As can also be observed in \cref{fig:5cellProblems}, in this case, the symmetry properties of the geometry and the particular boundary conditions yield the following relations between the cell solutions $q_1^\pm$ and $q_2^\pm$. Let $R\in \R^{3\times 3}$ be the rotation such that $Re_1 = e_2$, $Re_2 = -e_1$ and $Re_3 = e_3$. Then for $y\in Z_f$, it holds
\begin{align} \label{eq:rel_cell_sol_1}
    R^T q_1^{\pm} (Ry)= -q_2^\pm(y) \quad  \mbox{ and } \quad R^T q_2^{\pm} (Ry) = q_1^\pm (y).
\end{align}
To show \eqref{eq:rel_cell_sol_1}, let us define 
\begin{align*}
    \tilde{q}_i^{\pm}(y):= R^T q_i^{\pm} (Ry), \qquad \tilde{\pi}_i^{\pm}(y):= \pi_i^{\pm}(Ry)
\end{align*}
for $i=1,2$. An elemental calculation shows
\begin{align*}
   \nabla_y \cdot D_y(\tilde{q}_i^{\pm}) = R^T (\nabla_y \cdot D_y(q_i^{\pm}))(Ry), \qquad \nabla_y \tilde{\pi}_i^{\pm}= R^T (\nabla_y \pi_i^{\pm})(Ry).
\end{align*}
Further, we have the boundary conditions
\begin{align*}
    \tilde{q}_1^{\pm} (y) = R^T e_1 = -e_2, \quad \mbox{ for } y \in S^\pm,
\end{align*}
and similar $\tilde{q}_2^{\pm}(y) = e_1$ on $S^\pm$. By the uniqueness of the cell problems we obtain 
\begin{align*}
\tilde{q}_1^{\pm} = -q_2^{\pm} \quad \mbox{ and } \quad \tilde{q}_2^{\pm} = q_1^{\pm}.
\end{align*}
Further relations can be established between $q_i^+$ and $q_i^-$ for $i=1,2$. 
If we consider the rotation $\hat{R}\in \R^{3\times 3}$ such that 
$\hat{R}e_1 = -e_1$, $\hat{R}e_2 = e_2$ and $\hat{R}e_3 = -e_3$, then it holds
\begin{align} \label{eq:rel_cell_sol_2}
    \hat{R}^T q_1^\pm (\hat{R}y)= -q_1^\mp (y) \quad  \mbox{ and } \quad \hat{R}^T q_2^\pm(\hat{R}y) = q_2^\mp (y), \quad \mbox{for } y\in Z_f.
\end{align}
The proof is analogous to that of \eqref{eq:rel_cell_sol_1}.

Since due to the relations \eqref{eq:rel_cell_sol_1}-\eqref{eq:rel_cell_sol_2} solutions $q_1^-,\,q_2^\pm$ exhibit a similar behavior to that of $q_1^+$, in the following we will analyze the cell solutions $q_1^+$ and $q_3$.
We first present these solutions for three different heights and fixed radius $r=0.1$, and then fix $h=1.7$ and show solutions for three different radii. The corresponding results are displayed in  \cref{fig:cellSolutions_manyH} and \cref{fig:cellSolutions_manyR}, respectively. 
\begin{figure}[H]
    \begin{minipage}[t]{0.8\textwidth}
    \vspace{0pt}
    \subfloat[]{\includegraphics[width=0.305\linewidth]{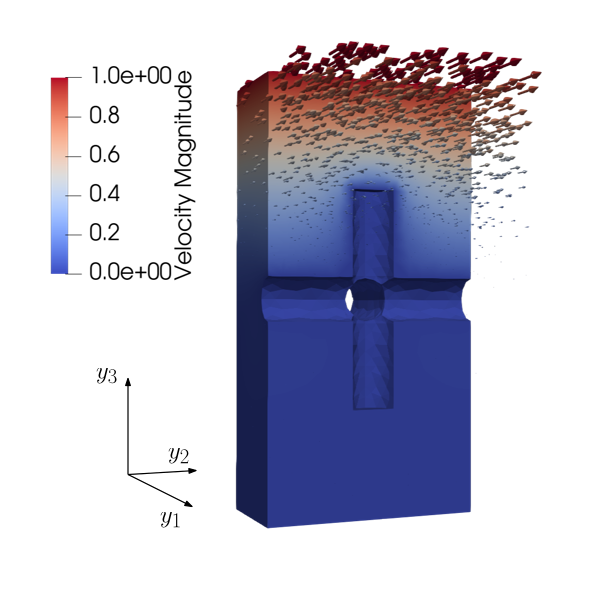}}
    \subfloat[]{\includegraphics[width=0.305\linewidth]{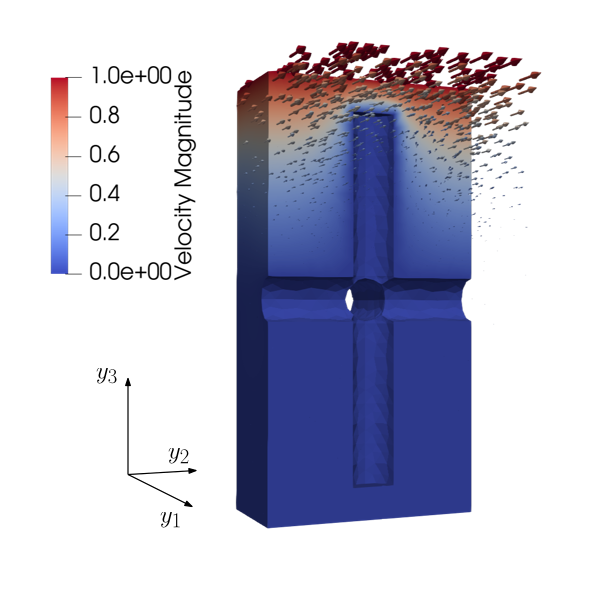}}
    \subfloat[]{\includegraphics[width=0.305\linewidth]{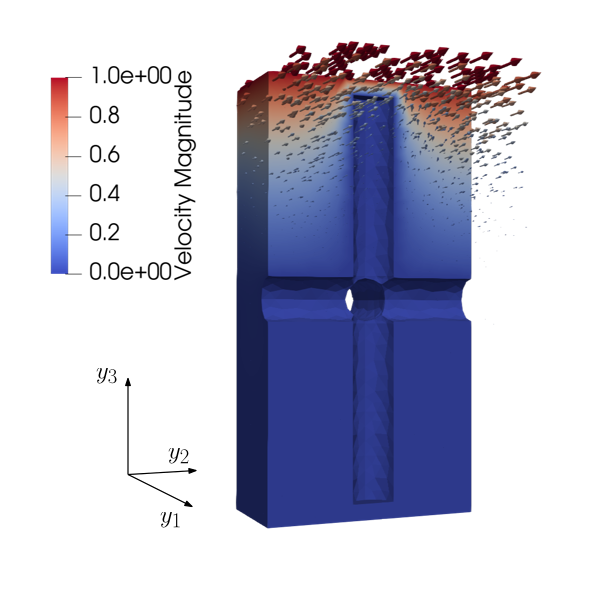}}\\
    \subfloat[]{\includegraphics[width=0.305\linewidth]{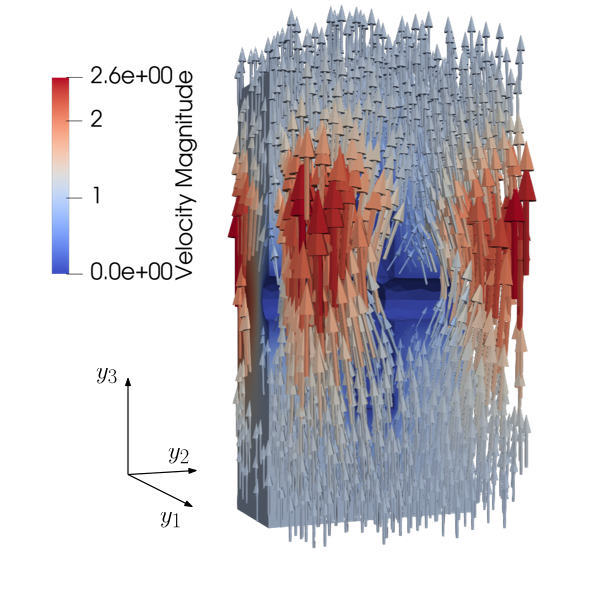}}
    \subfloat[]{\includegraphics[width=0.305\linewidth]{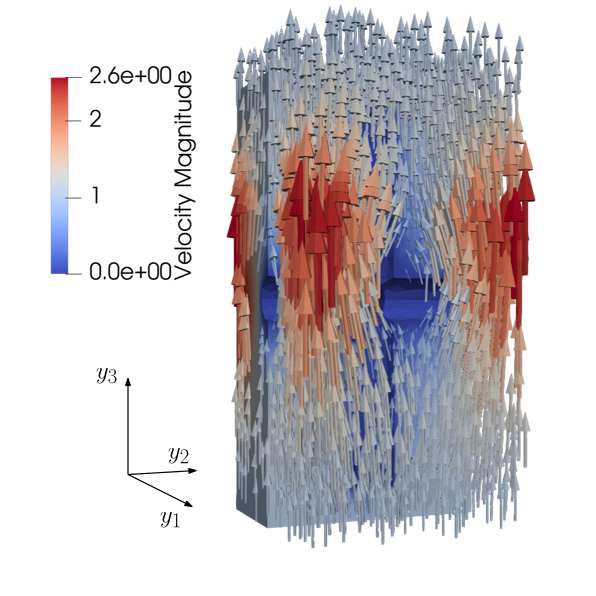}}
    \subfloat[]{\includegraphics[width=0.305\linewidth]{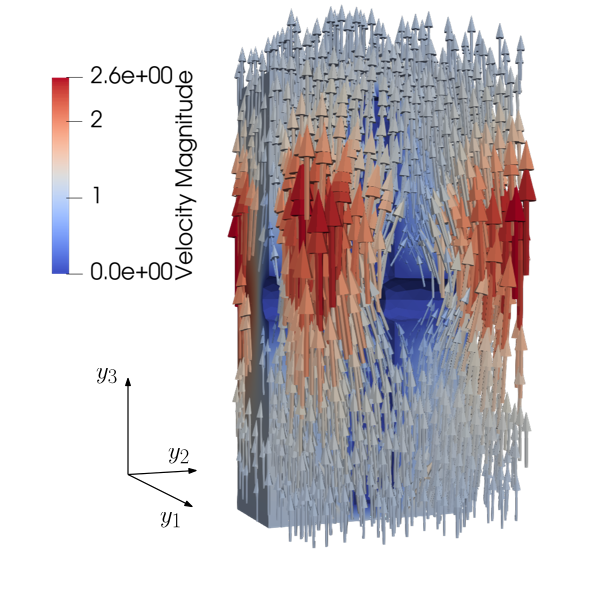}}
    \end{minipage}\hfill
\begin{minipage}[t]{0.2\textwidth}
    \vspace{10pt}
    \caption{ 
    Solutions $q_1^+$ in (a)-(c) and $q_3$ in (d)-(e) to Stokes cell problems \eqref{prob:cellprob1-4} and \eqref{prob:cellprob5}, respectively, for values $h=1.0$, $1.7$, $1.85$, and fixed $r=0.1$.}
    \label{fig:cellSolutions_manyH}
\end{minipage}
\end{figure}
\begin{figure}[H]
    \begin{minipage}[t]{0.8\textwidth}
    \vspace{0pt}
    \subfloat[]{\includegraphics[width=0.305\linewidth]{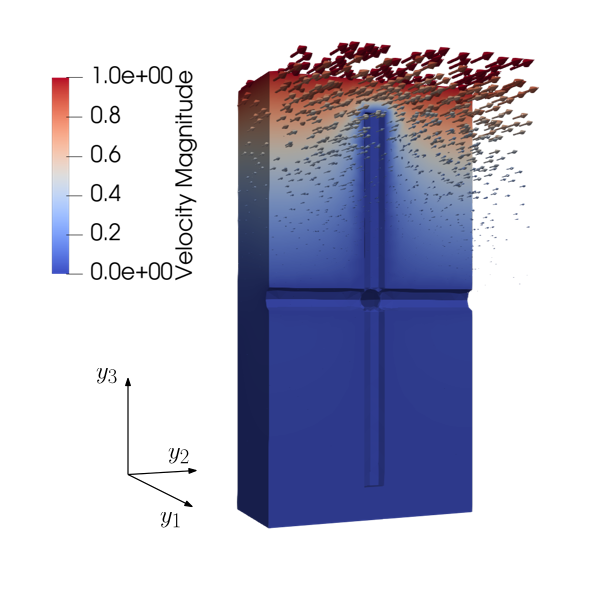}}
    \subfloat[]{\includegraphics[width=0.305\linewidth]{Images/Simulations/cell_sol/symm_cases/1d_height1_7_r0_1_cellProb1_w-y.png}}
    \subfloat[]{\includegraphics[width=0.305\linewidth]{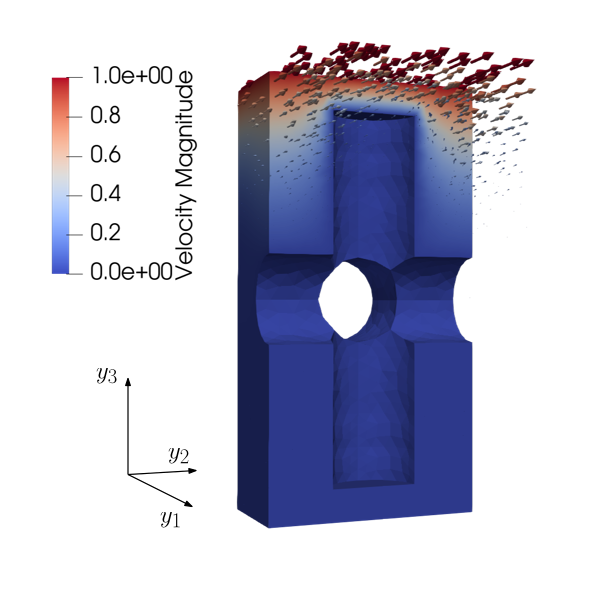}}\\
    \subfloat[]{\includegraphics[width=0.305\linewidth]{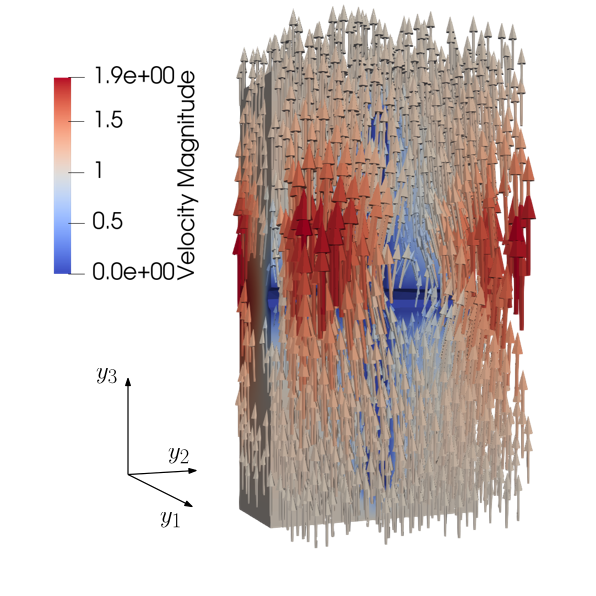}}
    \subfloat[]{\includegraphics[width=0.305\linewidth]{Images/Simulations/cell_sol/symm_cases/1d_height1_7_r0_1_cellProb5_w-y.png}}
    \subfloat[]{\includegraphics[width=0.305\linewidth]{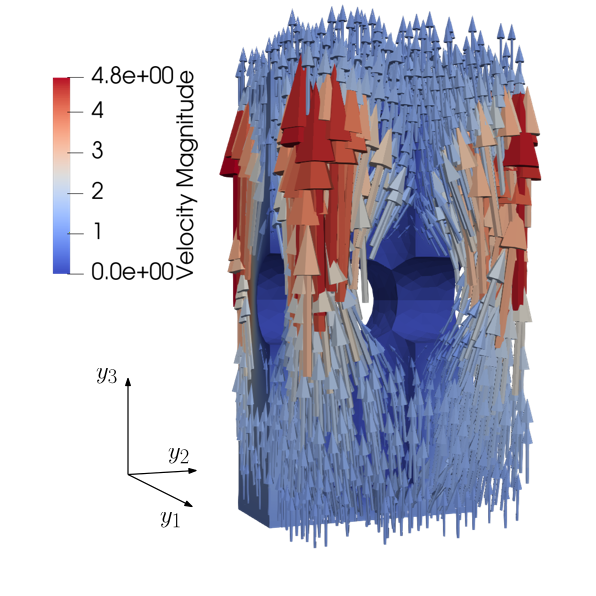}}
    \end{minipage}\hfill
\begin{minipage}[t]{0.2\textwidth}
\vspace{10pt}
\caption{ 
Solutions $q_1^+$ in (a)-(c) and $q_3$ in (d)-(e) to Stokes cell problems \eqref{prob:cellprob1-4} and \eqref{prob:cellprob5}, respectively, for $r=0.05$, $0.1$, $0.2$, and fixed $h=1.7$.}
\label{fig:cellSolutions_manyR}
\end{minipage}
\end{figure}
\cref{fig:cellSolutions_manyH} and \cref{fig:cellSolutions_manyR} show that the numerical method is indeed capable of solving the Stokes problem independently of the proposed changes in the geometry. In particular, we can observe that cell problem $q_1^+$ exhibits visible differences for varying $h,\,r$, while $q_3$ is rather unchanged from increasing $h$ and shows sharper gradients for increasing $r$.

\begin{figure}[H]
    \begin{minipage}[t]{0.8\textwidth}\centering
    \vspace{0pt}
       \subfloat[]{
        \includegraphics[width=0.268\linewidth]{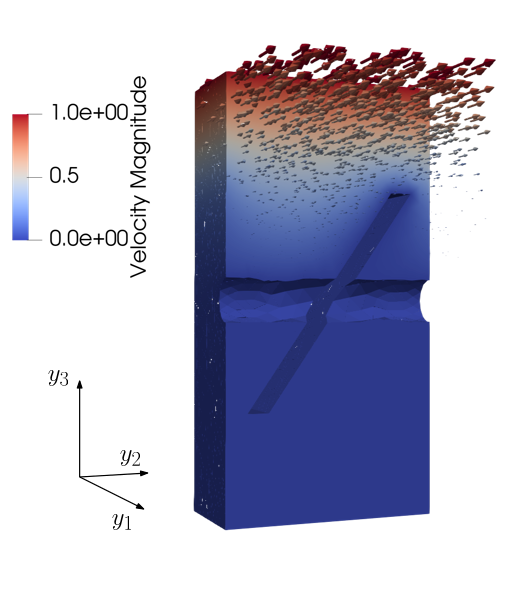}}\hspace{0.28cm}
    \subfloat[]{
        \includegraphics[width=0.268\linewidth]{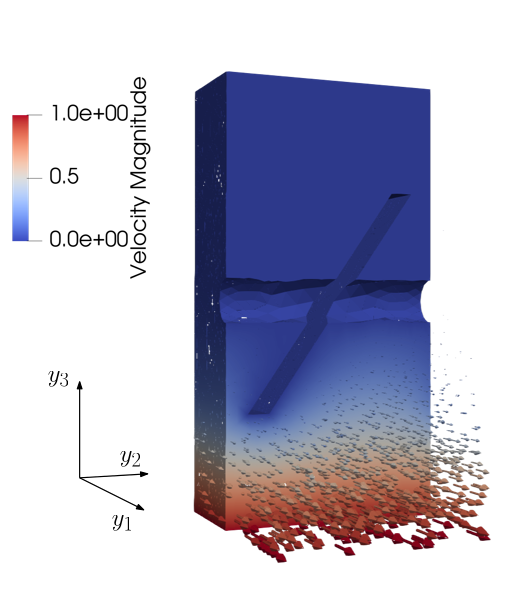}}\hspace{0.28cm}
    \subfloat[]{
        \includegraphics[width=0.268\linewidth]{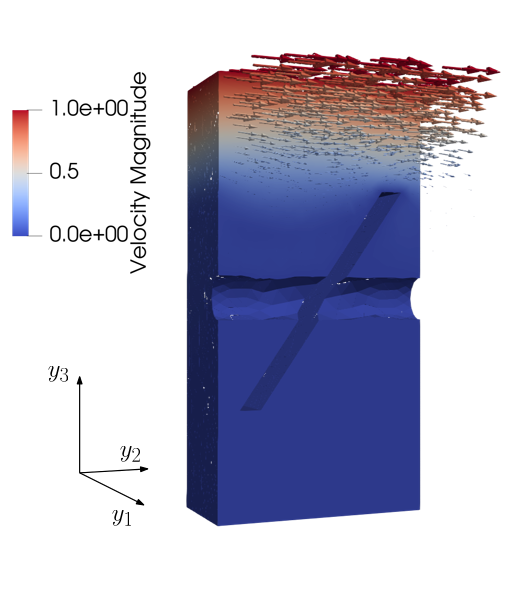}}\hspace{0.28cm}
    \subfloat[]{
        \includegraphics[width=0.268\linewidth]{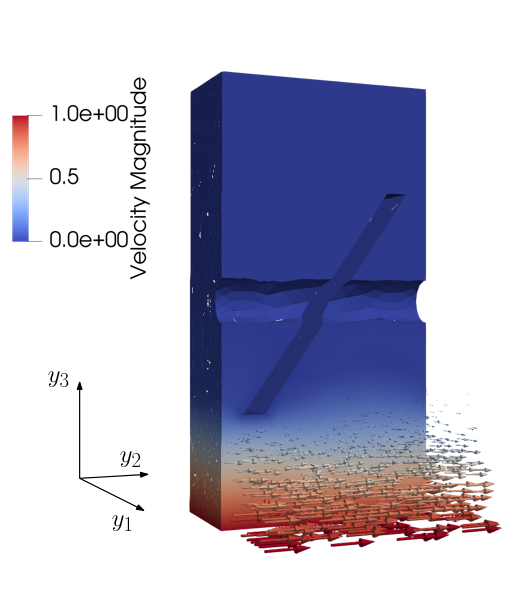}}\hspace{0.28cm}
    \subfloat[]{
        \includegraphics[width=0.284\linewidth]{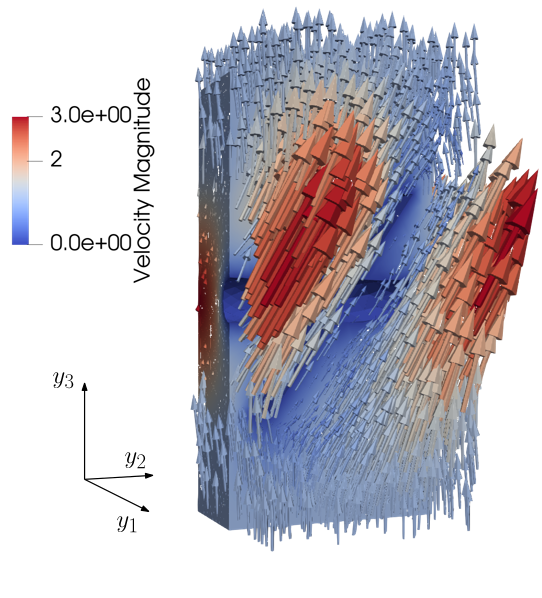}}
\end{minipage}\hfill
\begin{minipage}[t]{0.2\textwidth}
    \vspace{10pt}
    \caption{Sectional views of solutions $q_i^\pm, \,i\in \{1,2\}$, in (a-d) and $q_3$ in (e) to Stokes cell problems \eqref{prob:cellprob1-4} and \eqref{prob:cellprob5}, respectively, in the case of asymmetric cell geometry.}
    \label{fig:5cellProblems_asymmetric}
\end{minipage}
\end{figure}
In \cref{fig:5cellProblems_asymmetric}, the cell solutions are computed for the asymmetric obstacle given in \cref{fig:PeriodicityCell-sym_asym}(c). While the cell solutions $q^\pm_i,\, i \in \{1,2\}$, look qualitatively similar to the previously computed cell solutions for the symmetric obstacle, see \cref{fig:5cellProblems}(a)-(d), the velocity $q_3$ now exhibits a visible deflection at the inclined plane. It will be investigated in Section~\ref{sec:macroscopicSimulations} how different cell geometries affect the macroscopic flow.

\subsection{Computation of effective coefficients and sensitivity study} \label{subsec:effectivecoefficients}
In the following, based on the numerical results for the cell problems, we compute the effective coefficient tensors $K^\pm, \ M^\pm \in \mathbb{R}^{3\times 3}$  defined in \eqref{Effective_coeff}. For the case when the microscopic geometry is generated by the symmetric cylindrical cross, we analyze their sensitivity to the geometric parameters. 

Let us first mention some properties of the tensors $K^\pm$ and $M^\pm$, which can be derived analytically. Firstly, by definition, $K^\pm$ is symmetric, and $K_{33}^+ = K_{33}^-$. Furthermore, for the symmetric geometry mentioned above, the relations  \eqref{eq:rel_cell_sol_1} for the cell solutions lead to the following properties:
\begin{align}\label{eq:rel_eff_coef_1}
 K_{13}^{\pm} = K_{23}^{\pm} = 0, \quad  K_{12}^{\pm} = 0 \quad \mbox{ and } \quad K_{11}^\pm = K_{22}^{\pm}.
\end{align}
As an example, let us show the first relation in \eqref{eq:rel_eff_coef_1}. Using relation  \eqref{eq:rel_cell_sol_1} and a change of coordinates, we get for $\alpha \in \{\pm\}$
\begin{align*}
 K_{13}^\alpha &= \int_{Z_f} D_y(q_1^\alpha)(y) : D_y(q_3)(y) dy 
 \\ &= \int_{Z_f} \left( R^T D_y(q_2^\alpha)(Ry) R \right) : \left( R^T D_y(q_3)(Ry) R \right)  dy 
 \\ 
 &= \int_{Z_f} D_y(q_2^\alpha)(Ry) : D_y(q_3)(Ry)  dy 
 \\
 &= \int_{Z_f} D_{\tilde{y}}(q_2^\alpha)(\tilde{y}) : D_{\tilde{y}}(q_3)(\tilde{y})  d\tilde{y} = K_{23}^\alpha 
\end{align*}
Similarly, we obtain $K_{23}^{\pm} = -K_{13}^{\pm}$. This implies the desired result. The other relations in \eqref{eq:rel_eff_coef_1} follow by similar arguments. Using now relation \eqref{eq:rel_cell_sol_2}, we can show that $K_{ii}^+ = K_{ii}^-$, for $i=1,2$. Thus, we conclude that $K^+ = K^-$ are diagonal tensors with two different entries, namely $K_{11}^+=K_{22}^+$ and $K_{33}^+$. Concerning the tensors $M^\pm$, by definition, we have $M^- =(M^+)^T$. By similar arguments as above, we obtain that $M^\pm$ is diagonal with $M_{11}^\pm = M_{22}^\pm$.

The coefficients $K^\pm$ and $M^\pm$ are computed numerically for the following values of height $h$ and radius $r$:
{\small\begin{align*}h = 1.0, 1.05, 1.1, ..., 1.85, 1.86, 1.87, ..., 1.99 \quad\text{ and }\quad r=0.05,0.1,0.2.
\end{align*}}\unskip\nobreak 
The numerical results confirm the properties of the coefficients established analytically above. Furthermore, it turns out that $M^\pm_{11}= M^\pm_{22}$ is almost vanishing for all values of $h$.  
\begin{figure}[H]
    \centering
    \includegraphics[width=1.0\linewidth]{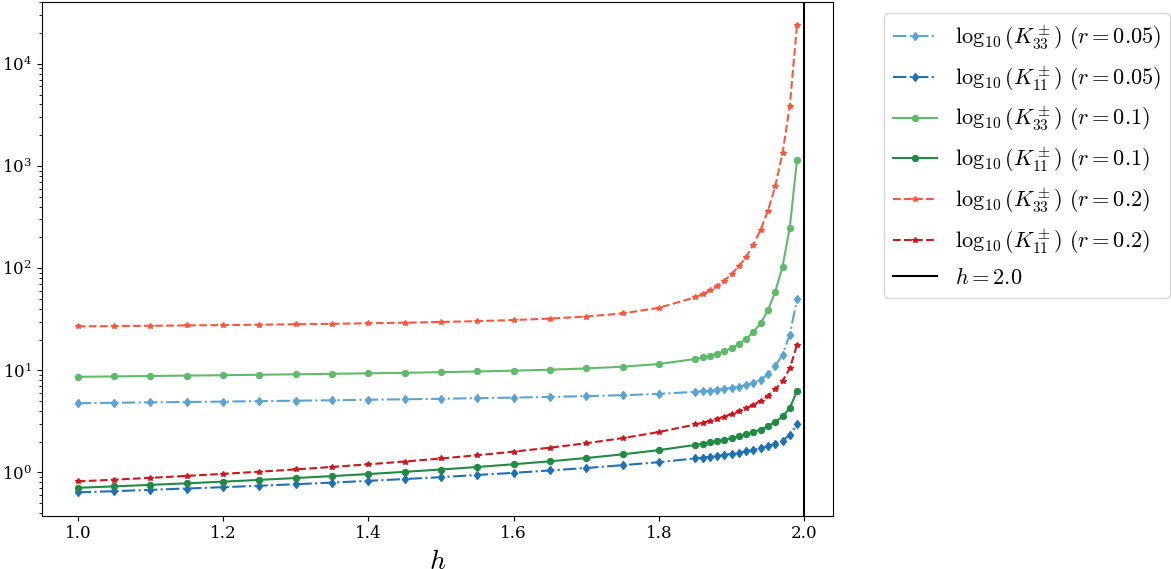}
    \caption{Coefficients $K^\pm_{11}$ and $K^\pm_{33}$ computed for different heights ($1.0 \leq h \leq 1.99$) and three radii ($r = 0.05,\,0.1,\,0.2$), shown on a logarithmic scale.}
    \label{fig:Coeff_Sim_Res}
\end{figure}
The (non-vanishing) coefficients $K^\pm_{11}$ and $K^\pm_{33}$ are plotted in \cref{fig:Coeff_Sim_Res}. We see that the coefficients $K^\pm_{33}$ consistently show larger magnitude across all heights and radii. Furthermore, we observe that the coefficients $K^\pm_{33}$ remain nearly constant up to $h = 1.85$. This corresponds to the observations in \cref{fig:cellSolutions_manyH}, (d)-(f), where no visible change in the cell solutions $q_3$ was observed for the heights $1.0, 1.7$ and $1.85$. We also note that for values of $h$ greater than $1.85$, there is a sharp increase in the coefficients $K^\pm_{11}$ and $K^\pm_{33}$, and that for $h \to 2$, the coefficients appear to diverge. Finally, we can see that for increasing values of $r$, the values of all coefficients increase.
\begin{remark}\label{rmk:vanishingM}
The almost vanishing values for the coefficient $M^\pm_{ii},\, i=1,2,$ are consistent with the localization of the energy of the cell solutions $q^\pm_i,\, i=1,2$ near the boundary $S^\pm$ where nonzero Dirichlet boundary conditions are imposed, see \cref{fig:5cellProblems}. Since the localization of the energies also holds for the asymmetric cell geometry, see \cref{fig:5cellProblems_asymmetric}, we computed the corresponding coefficients also for this case and obtained a similar behavior. In order to further understand the influence of the homogeneous Dirichlet boundary condition at the boundary $\Gamma$ of the solid obstacle on the values of the coefficients, we computed the coefficient $M_{11}^+$ with a solid obstacle consisting of a sphere centered at the origin. We then let its radius $r_{\text{sphere}}$ vanish. The results are shown in \cref{fig:MSphere}. We can observe that for a rather large spherical obstacle, $M^+_{11}$ is almost zero. As the sphere shrinks, $M^+_{11}$ converges to $-\frac{1}{4}$, which corresponds to the value of the coefficient $M^+_{11}$ for the case of the cell without an obstacle, see \eqref{eq:coeff_empty}.
\begin{figure}[H]
    \centering
    \includegraphics[width=1.0\linewidth]{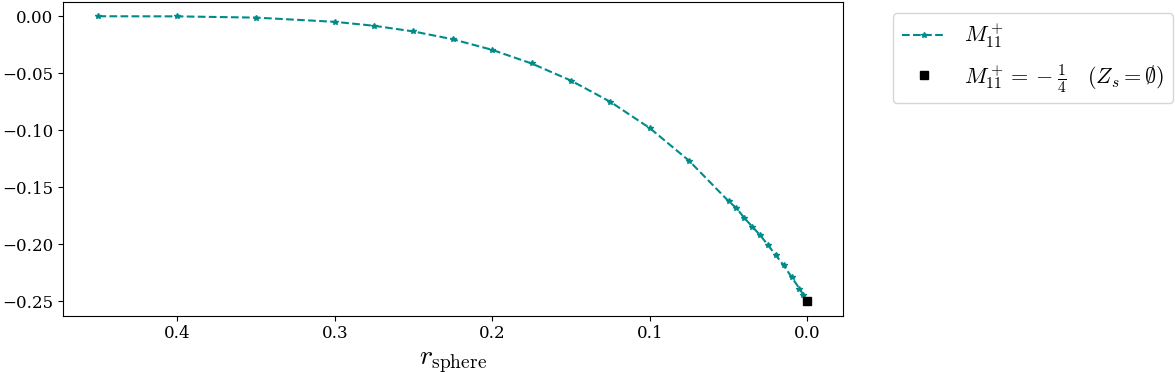}
    \caption{Coefficient $M^+_{11}$ computed for vanishing radii $r_{\text{sphere}}$ of a spherical obstacle centered at the origin in the periodicity cell $Z$.}
    \label{fig:MSphere}
\end{figure}
\end{remark}

\section{Numerical investigation of the effective transmission model}
\label{sec:macroscopicSimulations}
In this section, we start with the formulation of the discretized  effective transmission model and prove its well-posedness. We then perform numerical simulations to 
understand how the solutions corresponding to the different microstructures, see \cref{fig:PeriodicityCell-sym_asym}(b) and (c), (and thus to different effective coefficients) behave, particularly in the region near the interface. We also compare this behavior to the case where there is no solid skeleton in the layer, like in \cref{fig:PeriodicityCell-sym_asym}(a). Last but not least, we aim to investigate numerically to which extent the effective transmission problem approximates the microscopic problem which contains a thin porous layer of finite thickness.
\subsection{Discretization and numerical analysis of the effective transmission model}
\label{sec:discr_transm_model} 
We discretize the effective model \eqref{prob:macroscopic} 
using a mixed finite-element discretization with
Taylor--Hood-type ansatz spaces. Additionally to the
boundary conditions \eqref{eq:bound_cond_Dir_Neumann}
we allow for periodic boundary conditions of the form
\begin{align*}
    v^\pm(\Phi^\pm(x)) &= v^\pm(x)
     &\text{ on } \partial_P\Omega^\pm,
\end{align*}
with the boundary segment
$\partial_P\Omega^\pm \subset (\partial \Omega^\pm \setminus \Sigma)$.
Here $\Phi^\pm:\partial_P\Omega^\pm \to (\partial \Omega^\pm \setminus \Sigma)$
is a translation such that the outer boundary
of $\Omega^\pm$ is given by the disjoint decomposition
\begin{align}
    \label{eq:bc_with_periodic}
  \partial \Omega^\pm \setminus \Sigma
  = \partial_D\Omega^\pm \,\dot{\cup}\, \partial_N\Omega^\pm \,\dot{\cup}\,
  \partial_P\Omega^\pm \,\dot{\cup}\, \Phi^\pm(\partial_P\Omega^\pm).
\end{align}
Notice that this only covers cases where both sides
$\partial_P\Omega^\pm$ and $\Phi^\pm(\partial_P\Omega^\pm)$
of each periodic boundary are in the same subdomain $\Omega^\pm$.

While we could use a mixed finite element discretization
using classical Taylor--Hood ansatz functions on the
whole domain $\Omega$, the expected discontinuity of
the tangential velocities and the pressure across
the interface $\Sigma$ would require very fine
meshes for an appropriate approximation.
As a remedy we use Taylor--Hood
spaces separately on each subdomain and only enforce continuity
of the normal velocities across $\Sigma$, which
allows to represent the expected discontinuities
at the discrete level.

The velocity space with the corresponding continuity and
homogeneous boundary conditions is then given by
\begin{align*}
    H = \{\phi \in H^+ \times H^-\;:\; [\phi^+]_3 = [\phi^-]_3 \text{ on }\Sigma \}
\end{align*}
where
\begin{align*}
    H^\pm = \{\phi \in H^1(\Omega^\pm)\;:\; v= 0 \text{ on }\partial_D \Omega^\pm,\,
    v^\pm(\Phi^\pm(x)) = v^\pm(x) \text{ on } \partial_P\Omega^\pm\}
\end{align*}
now additionally incorporates periodic boundary conditions for $\Omega^\pm$.

For the discretization we consider uniform conforming
rectangular hexahedral grids $\mathcal{T}_h$
on $\Omega$ that are compatible with $\Omega^+$ and
$\Omega^-$ in the sense that both subdomains are exactly
covered by disjoint subsets of elements in $\mathcal{T}_h$.
Furthermore we assume that each of the boundary segments
in \eqref{eq:bc_with_periodic} is resolved by grid elements.
Here $h>0$ denotes the (maximal) edge length of the elements
in $\mathcal{T}_h$.

On such grids we denote by
\begin{align*}
    S_{h,k}^\pm
        &= \{ \phi \in C(\Omega^\pm) \,:\, \phi|_e \in \mathcal{Q}_{k}\forall e \in \mathcal{T}_h
        \subset H^1(\Omega^\pm)\}
\end{align*}
the classical Lagrange finite element spaces
on the subdomains $\Omega^\pm$
where $\mathcal{Q}_k$ are the $3$-variate tensor
polynomials of order $k\geq 1$.
Using this notation the product space
\begin{align*}
    (S_{h,k+1}^\pm)^3 &\times S_{h,k}^\pm
\end{align*}
is the classical Taylor--Hood-pair which is well-known
to be inf-sup-stable with $h$-independent constants
for the divergence operator.
The same holds true for the pair
\begin{align}
    \label{eq:th_classical}
    \bigl((S_{h,k+1}^\pm)^3 \cap H^1(\Omega^\pm, \Gamma)^3\bigr)
        &\times S_{h,k}^\pm
\end{align}
incorporating Dirichlet boundary conditions
for any $h$-independent strict subset
$\Gamma \subset \partial \Omega^\pm$ which is resolved by the grid elements
and the pair
\begin{align}
    \label{eq:th_classical_diri}
    \bigl((S_{h,k+1}^\pm)^3 \cap H^1(\Omega^\pm, \partial \Omega^\pm)^3\bigr)
        &\times (S_{h,k}^\pm \cap L^2_0(\Omega^\pm))
\end{align}
incorporating Dirichlet boundary conditions on the whole subdomain boundary.

Next we define discrete subspaces with the desired continuity conditions
\begin{align*}
    H_h^\pm &= (S_{h,k+1}^\pm)^3 \cap H^\pm,\\
    H_h &
        = (S_{h,k+1}^+)^3 \times (S_{h,k+1}^-)^3 \cap H
        = H^+_h \times H^-_h \cap H,\\
    Q_h &= (S_{h,k}^+ \times S_{h,k}^-) \cap Q.
\end{align*}
In contrast to the classical Taylor--Hood-pair on $\Omega$,
the functions in $H_h\times Q_h$ have discontinuous tangential velocities
and discontinuous pressure across the interface $\Sigma$.

Given these discrete function spaces we define the Galerkin
discretization of \eqref{eq:weak_problem} by:
Find $(v_h-v_D, p_h) \in H_h \times Q_h$ such that
\begin{subequations}
\label{eq:weak_problem_disc}
\begin{alignat}{3}
    &a(v_h,\phi)&{} + b(\phi,p_h) &= \ell(\phi)
        \qquad&\forall \phi \in H_h,\\
    &b(v_h,q)& &= 0
        & \forall q \in Q_h.
\end{alignat}
\end{subequations}
For simplicity we assume that
the boundary values can be represented exactly in 
the finite element space, i.e.,
$v_D \in (S_{h,k+1}^+)^3 \times (S_{h,k+1}^-)^3$.

To show existence and uniqueness we proceed analogously
as in the continuous case.
First we show inf-sup-stability of the pair $H_h \times Q_h$.

\begin{lemma}
    Assume that the grid  $\mathcal{T}_h$ is given by a uniform refinement
    of an initial coarse grid $\mathcal{T}_{h_0}$. Then the pair
    $H_h \times Q_h$ is inf-sup-stable for the bilinear form
    $b(\cdot,\cdot)$, i.e., there is a constant $\beta>0$
    independent of $h$, but possibly dependent on the coarse mesh,
    such that
    \begin{align*}
        \forall p \in Q_h:
        \qquad
            \sup_{\substack{v \in H_h\\v \neq 0}}
            \frac{b(v, p)}{\|v\|}
        \geq \beta \|p\|_{L^2(\Omega)}.
    \end{align*}
\end{lemma}
\begin{proof}
    We first consider the case that
    $\partial_P\Omega^+ = \partial_P\Omega^- = \emptyset$.
    Then the proof is almost literally the same as the proof
    of the continuous inf-sup-condition for
    Proposition~\ref{prop:weak_problem_existence}.
    We only have to replace the spaces
    $H,\,Q,\, H^\pm,\, L^2(\Omega^\pm)$
    by their discrete analogues
    $H_h,\, Q_h,\, H^\pm_h,\, S_{h,k}^\pm$, respectively.
    Furthermore, we have to replace
    the continuous weak divergence operator
    $\operatorname{div} :  H^1(\Omega^\pm)^3 \to L^2(\Omega^\pm)$
    by the discrete weak divergence operator
    $\operatorname{div}_h :  H^\pm_h \to S_{h,k}^\pm$
    which is induced by the bilinear form
    $b(\cdot,\cdot) : H^\pm_h \times S_{h,k}^\pm \to \mathbb{R}$ 
    and use the inf-sup-conditions on the pairs
    \eqref{eq:th_classical} and \eqref{eq:th_classical_diri}
    instead of their continuous analogues.
    
    The obtained inf-sup-constant $\beta>0$
    on $H_h \times Q_h$ then only depends on the
    $h$-independent constants for the pairs
    \eqref{eq:th_classical} and \eqref{eq:th_classical_diri}
    and possibly on the terms
    $\hat{\beta} = |\frac{b(\hat{v}, \hat{p})}{\|\hat{v}\|\|\hat{p}\|}|>0$.
    Since the functions $\hat{p}$ from the proof of Proposition~\ref{prop:weak_problem_existence} are piecewise constant 
    with respect to $\Omega^\pm$,
    they are contained in the pressure space on the coarse grid
    $\mathcal{T}_{h_0}$.
    Hence we can also construct the corresponding functions $\hat{v}$
    in the velocity space on the coarse grid $\mathcal{T}_{h_0}$,
    such that $\hat{\beta}$ and thus $\beta>0$ are
    $h$-independent.

    In the case of nontrivial periodic boundary conditions
    we can use the above argument to show an $h$-independent
    inf-sup-estimate on $\tilde{H}_h \times Q_h$ where
    \begin{align*}
        \tilde{H}_h
            = \{v \in H_h\;:\; v = 0 \text{ on }
                \partial_P\Omega^+ \cup \Phi^+(\partial_P\Omega^+)
                \cup \partial_P\Omega^- \cup \Phi^-(\partial_P\Omega^-)
                \}
    \end{align*}
    is the space with homogeneous Dirichlet conditions on the periodic
    boundaries. Since $H_h \supset \tilde{H}_h$ is a super-space,
    this estimate directly carries over to $H_h \times Q_h$.
\end{proof}

\begin{proposition}
    \label{prop:disc_existence}
    Let $Z_s \neq \emptyset$.
    Then, there exists a unique weak solution $(v_h,p_h)$ to the discrete
    transmission problem~\eqref{eq:weak_problem_disc}. Furthermore there
    is a constant $C>0$ independent of $v$, $p$, and $h$ such that
    \begin{align}
        \|v-v_h\|_H + \|p-p_h\|_{L^2(\Omega)}
        \leq C \Bigl(
            \inf_{\phi \in H_h} \|v-\phi \|_H +
            \inf_{q \in Q_h} \|p-q \|_{L^2(\Omega)}
        \Bigr).
    \end{align}
\end{proposition}

\begin{proof}
    We only need to note that continuity and coercivity
    of $a(\cdot, \cdot)$ and continuity of $b(\cdot,\cdot)$
    on the continuous spaces are inherited by the discrete
    subspaces. Furthermore the inf-sup-condition is satisfied
    with an $h$-independent constant. Then the existence
    result and error estimate follow from the classical
    theory of Brezzi \cite{Brezzi1974}.
\end{proof}

As usual combining the best-approximation error bounds
of Proposition~\ref{prop:disc_existence} with interpolation
error bounds results in the error estimate
\begin{align}
    \|v-v_h\|_H + \|p-p_h\|_{L^2(\Omega)} \in O(h^{k+1})
\end{align}
if the solution $(v,p)$ is sufficiently smooth.

\subsection{Implementation aspects}
\label{sec:ImplementationAspects}
As for the cell-problem, the implementation of the
macroscopic transmission problem is also carried out using
\textit{DUNE} \cite{DUNE1, DUNE:2021}.
A key challenge of the finite element discretization
introduced above is that,
in contrast to the classical Taylor--Hood-space,
tangential velocity and pressure are discontinuous
at the interface $\Sigma$.
To implement this space we use the \textit{DUNE} module
dune-functions \cite{EngwerEtAl2025} which provides
a flexible framework for the construction of
nontrivial finite element spaces.
Here we in particular make use of its ability to
construct nested product spaces and
spaces constrained to subdomains.
After defining objects \texttt{omegaPlus} and \texttt{omegaMinus}
which describe the two subdomains,
a basis for the piecewise Taylor--Hood-like space
\begin{align*}
    \Bigl((S_{h,k+1}^+)^{3}\times(S_{h,k+1}^-)^{3}\Bigr)
    \times
    \Bigl((S_{h,k}^+)\times(S_{h,k}^-)\Bigr)
\end{align*}
can be constructed using:
\begin{lstlisting}[language=C++]
auto basis = makeBasis(
      gridView,
      composite(
        composite(
          restrict(power<3>(lagrange<k+1>()), omegaPlus),
          restrict(power<3>(lagrange<k+1>()), omegaMinus)
        ),
        composite(
          restrict(lagrange<k>(), omegaPlus),
          restrict(lagrange<k>(), omegaMinus)
        )
      )
    );
\end{lstlisting}
Here, the \texttt{gridView} represents the computational grid,
\texttt{lagrange<k>()} is the $k$-th order Lagrange finite element space,
\texttt{restrict(...)} constraints a space to a subdomain,
\texttt{composite(...)} is a product of spaces,
\texttt{power<3>(...)} is a 3-fold power of a space.
In the actual implementation we pass additional arguments
that guide how the respective basis functions of this
product space are indexed. This allows to use an indexing
that is suited for the implementation of a
block-preconditioner for the linear saddle-point problem.
Notice that this space is discontinuous across
$\Sigma$ and thus not yet the space we want to use.

The discretization in this space is implemented using
the dune-fufem \cite{fufem} module which allows
to assemble matrices and vectors by defining the
respective bilinear and linear forms directly in \texttt{C++}
in a straightforward way without the need to
manually implement assemblers or to use code generation.
The essential Dirichlet and periodic boundary conditions
can be enforced using an abstract constraints mechanism
provided by dune-fufem as described in Section~\ref{subsec:meshGen}.
The functionality for periodic
boundary conditions can also be used for internal
interfaces and to couple different components of
the product space and thus allows to enforce
the continuity of normal velocities
$[v^+]_3 = [v^-]_3$ on $\Sigma$.

We point out that the implementation of the
\texttt{restrict(...)} feature in dune-functions, which
allows to restrict a basis to a subdomain as well as the support
in dune-fufem to assemble variational forms for such 
bases and on the interface between subdomains
was pushed forward by the present application.

\subsection{Numerical convergence study}
In the following, we perform a numerical convergence study with respect to the spatial refinement of the macroscopic grid. Hereby, we consider the computational domain $\Omega$ with $H = 1$ and
\begin{align}\label{eq:Omega_sim_1}
\Sigma := \left(-\frac12, \frac12 \right) \times \left(-\frac12, \frac12 \right),
\end{align}
i.e., we consider
\begin{align}\label{eq:Omega_sim_2}
\Omega := \Sigma \times \left(-1, 1 \right), \quad \Omega^- := \Sigma \times \left(-1, 0 \right), \quad \Omega^+ := \Sigma \times \left(0, 1 \right).
\end{align} 
Starting from a coarse grid with two cubes representing $\Omega^+$ and $\Omega^-$, a sequence of successively refined grids is generated by uniform refinement, where each cube is subdivided into eight equal smaller cubes. We consider the piecewise smooth manufactured solution
\begin{align*}
    v^\pm(x) &= a^\pm + b^\pm x_3 + c^\pm x_3^3, \\
    p^\pm(x) &= d^\pm + q(x_1,x_2) + e^\pm x_3^2,
\end{align*}
with $a^\pm = (a_1^\pm, a_2^\pm, w)^T$, $b^\pm = \pm 2([K^\pm a^\pm]_t + M^\mp a^\mp)^T$, $c^\pm = (c_1^\pm, c_2^\pm, 0)^T$, $d^- = 0$,  $d^+ = -(K^+a^+ + K^- a^-)\cdot e_3$, and $q(x_1,x_2) = \sin(2\pi x_1)\sin(2\pi x_2)$, where $a_1^\pm$, $a_2^\pm$, $w$, $c_1^\pm$, $c_2^\pm$, $e^\pm \in \mathbb{R}$. These functions solve the effective transmission problem~\eqref{prob:macroscopic} with force terms
\begin{align*}
    f^\pm(x) &= -3 c^\pm x_3 + \nabla q(x_1,x_2) + 2e^\pm x_3 e_3 
\end{align*}
and boundary conditions on $\partial \Omega$ given by
\begin{align*}
    -[D(v^\pm)-p^\pm I]\nu^\pm &= g^\pm
    &\mbox{ on } \Gamma^\pm_\text{cap},\\[0.7ex]
    v^+, v^- &\text{ are } \Sigma\text{-periodic},
     \end{align*}
with boundary stresses
\begin{align*}
    g^\pm(x_1,x_2) &= \pm \frac{1}{2}(b^\pm + 3c^\pm) \mp p^\pm(x_1,x_2,\pm1)e_3.
\end{align*}
 For our study, we use the effective transmission coefficients obtained from the cylindrical cross cell geometry in \cref{fig:PeriodicityCell-sym_asym}(b). The resulting errors and estimated orders of convergence (EOC) are displayed in Table~\ref{tab:macroscopic-convergence}. After a brief pre-asymptotic phase, the behavior of the error follows the ideal convergence orders for first-order Taylor--Hood elements (corresponding to $k=1$), namely cubic for the $L^2$ error in the velocity and quadratic for the $L^2$ error in pressure, as well as quadratic in the $H^1$ errors for the velocities on the subdomains.
\begin{table}[H]
    \centering
    \caption{Convergence study for the effective transmission model.}
    \label{tab:macroscopic-convergence}
    \small
    \setlength{\tabcolsep}{3pt}

    \begin{tabular}{
        c
        r
        rr
        rr
        rr
        rr
    }
        \toprule
        \multirow{2}{*}{Level}
        & \multirow{2}{*}{DOFs}
        & \multicolumn{2}{c}{$\|v-v_h\|_{L^2(\Omega)}$}
        & \multicolumn{2}{c}{$\|p-p_h\|_{L^2(\Omega)}$}
        & \multicolumn{2}{c}{$\|v^+-v_h^+\|_{H^1(\Omega^+)}$}
        & \multicolumn{2}{c}{$\|v^--v_h^-\|_{H^1(\Omega^-)}$}
        \\
        \cmidrule(lr){3-4}
        \cmidrule(lr){5-6}
        \cmidrule(lr){7-8}
        \cmidrule(lr){9-10}

        & & Error & EOC & Error & EOC & Error & EOC & Error & EOC
        \\
        \midrule

        0
        & 178
        & $1.59\times10^{-1}$ & --
        & $7.16\times10^{-1}$ & --
        & $4.66\times10^{-1}$ & --
        & $9.33\times10^{-1}$ & --
        \\

        1
        & 804
        & $2.47\times10^{-2}$ & 2.69
        & $5.70\times10^{-1}$ & 0.33
        & $1.76\times10^{-1}$ & 1.41
        & $2.66\times10^{-1}$ & 1.81
        \\

        2
        & 4\,624
        & $3.18\times10^{-3}$ & 2.96
        & $1.08\times10^{-1}$ & 2.40
        & $4.00\times10^{-2}$ & 2.13
        & $6.40\times10^{-2}$ & 2.06
        \\

        3
        & 30\,936
        & $3.30\times10^{-4}$ & 3.27
        & $2.45\times10^{-2}$ & 2.15
        & $8.17\times10^{-3}$ & 2.29
        & $1.49\times10^{-2}$ & 2.10
        \\

        4
        & 225\,448
        & $3.93\times10^{-5}$ & 3.07
        & $5.86\times10^{-3}$ & 2.06
        & $1.87\times10^{-3}$ & 2.13
        & $3.64\times10^{-3}$ & 2.04
        \\

        5
        & 1\,719\,624
        & $4.87\times10^{-6}$ & 3.01
        & $1.45\times10^{-3}$ & 2.02
        & $4.55\times10^{-4}$ & 2.04
        & $9.03\times10^{-4}$ & 2.01
        \\

        \bottomrule
    \end{tabular}
\end{table}

\subsection{Simulation of the effective transmission model} 
\label{sec:simul_transm_model}
In this section, we perform three-dimensional numerical simulations of the effective transmission problem~\eqref{prob:macroscopic} for different effective coefficients, and for the boundary conditions at the outer boundary $\partial \Omega$ described in the scenarios A-D below. The computational domain $\Omega$ is defined in \eqref{eq:Omega_sim_1}-\eqref{eq:Omega_sim_2}.
The boundary of $\Omega$ is divided into the lateral boundary 
$$
\Gamma_\text{lat}:= \Gamma^\pm_\text{left} \cup \Gamma^\pm_\text{right} \cup \Gamma^\pm_\text{back} \cup \Gamma^\pm_\text{front}
$$
where
\begin{eqnarray*}
    && \Gamma^\pm_\text{left} := \partial \Omega^\pm \cap \{x_2= - \frac12 \}, \qquad \Gamma^\pm_\text{right} := \partial \Omega^\pm \cap \{ x_2=  \frac12 \},\\
    && \Gamma^\pm_\text{back} := \partial \Omega^\pm \cap \{x_1 = - \frac12 \},\qquad \Gamma^\pm_\text{front} := \partial \Omega^\pm \cap \{ x_1=  \frac12 \},
 \end{eqnarray*} 
 and the upper and lower boundary
$$
 \Gamma^\pm_\text{cap} := \Sigma \times \{ \pm 1 \}.
$$
We consider effective coefficients corresponding to three different cell geometries, namely a cell without a solid skeleton, a symmetric cell involving the cylindrical cross, and an asymmetric cell containing an inclined segment, respectively, see \cref{fig:PeriodicityCell-sym_asym}(a)-(c). Moreover, for the cell geometries with nonempty solid part $Z_s$,  we consider the boundary conditions from Scenarios A, B, C, and D below, while for the cell geometry with $Z_s= \emptyset$, the boundary conditions from Scenarios A, B, C are considered.

In \textbf{Scenario A}, a constant velocity is specified at the upper boundary, and a homogeneous stress boundary condition is imposed at the lower boundary of the domain $\Omega$. Periodic boundary conditions are specified at the lateral boundaries.
    \begin{align*}
    v^+ &= (0,0,-1)^T
    &\mbox{ on } \Gamma^+_\text{cap},\\[0.7ex]
    -[D(v^-)-p^-I]\nu^- &= (0,0,0)^T &\mbox{ on } \Gamma^-_\text{cap}, \\[0.7ex]
    v^+, v^- &\text{ are } \Sigma\text{-periodic}
\end{align*}

In \textbf{Scenario B}, a parabolic velocity profile is specified at the upper boundary, and a homogeneous stress boundary condition is imposed at the lower boundary of the domain $\Omega$. Homogeneous Dirichlet conditions for the velocity are imposed at the lateral boundaries.
    \begin{align*}
    v^+(x) &= (0, 0, -(1-4x_1^2)(1-4x_2^2))^T 
    &\mbox{ on } \Gamma^+_\text{cap}, \\[0.7ex]
    -[D(v^-)-p^-I]\nu^- &= (0,0,0)^T 
    &\mbox{ on } \Gamma^-_\text{cap},\\[0.7ex]
    v^\pm(x) &= (0, 0, 0)^T 
    &\mbox{ on } \Gamma_\text{lat}.
    \end{align*}

In \textbf{Scenario C}, the flow is driven by a parabolic inflow profile parallel to the interface $\Sigma$ at $\Gamma^+_\text{left}$, while on the other boundary parts homogeneous Dirichlet or stress boundary conditions are specified.
    \begin{align*}
    v^+ &= (0,0,0)^T
    &\mbox{ on } \Gamma^+_\text{cap},\\[0.7ex]
    -[D(v^-)-p^-I]\nu^- &= (0,0,0)^T &\mbox{ on } \Gamma^-_\text{cap}, \\[0.7ex]
    v^+(x) &= (0,  v^C_2(x), 0)^T  
    &\mbox{ on } \Gamma^+_\text{left}\\[0.7ex]
    -[D(v^+)-p^+I]\nu^+ &= (0,0,0)^T 
    &\mbox{ on } \Gamma^+_\text{right}\\[0.7ex]
    v^+ &= (0,0,0)^T
    &\mbox{ on } \Gamma^+_\text{front}\cup\Gamma^+_\text{back},\\[0.7ex]
    -[D(v^-)-p^-I]\nu^- &= (0,0,0)^T 
    &\mbox{ on }  \Gamma^-_\text{left} \cup \Gamma^-_\text{right} \cup \Gamma^-_\text{front} \cup \Gamma^-_\text{back},
    \end{align*}
    with
    $$
v^C_2(x) = \max\left(0,(1-4x_1^2)\left(-\frac{64}{9}x_3^2+\frac{80}{9}x_3-\frac{16}{9}\right)\right).
$$

In \textbf{Scenario D}, a constant normal stress is prescribed at the upper boundary, while a homogeneous stress boundary condition is imposed at the lower boundary. Periodic boundary conditions are specified at the lateral boundaries of the domain $\Omega$.
     \begin{align*}
    -[D(v^+)-p^+I]\nu^+ &= (0,0,1)^T 
    &\mbox{ on } \Gamma^+_\text{cap},\\[0.7ex]
    -[D(v^-)-p^-I]\nu^- &= (0,0,0)^T 
    &\mbox{ on } \Gamma^-_\text{cap},\\[0.7ex]
    v^+, v^- &\text{ are } \Sigma\text{-periodic}.
     \end{align*}
Notice that all four scenarios satisfy the assumptions of Proposition~\ref{prop:disc_existence} which guarantees
well-posedness and error estimates.
In particular all scenarios include natural boundary
conditions on a subset of the boundary leading to
unique solutions in the discrete pressure space
$S_{h,k}^+ \times S_{h,k}^-$ which allows for non-vanishing
mean pressure.

For all simulations, we assume that $f^\pm = (0,0,0)^T$. The main differences between the scenarios lie in the type of boundary condition driving the flow (inhomogeneous \textit{Dirichlet}: Scenarios A, B, C; inhomogeneous \textit{Neumann}: Scenario D) and the direction of the flow at the corresponding boundary (normal to the interface: Scenarios A, B, D; parallel to the interface: Scenario C).
\begin{figure}
    \centering
    \includegraphics[width=1.0\linewidth]{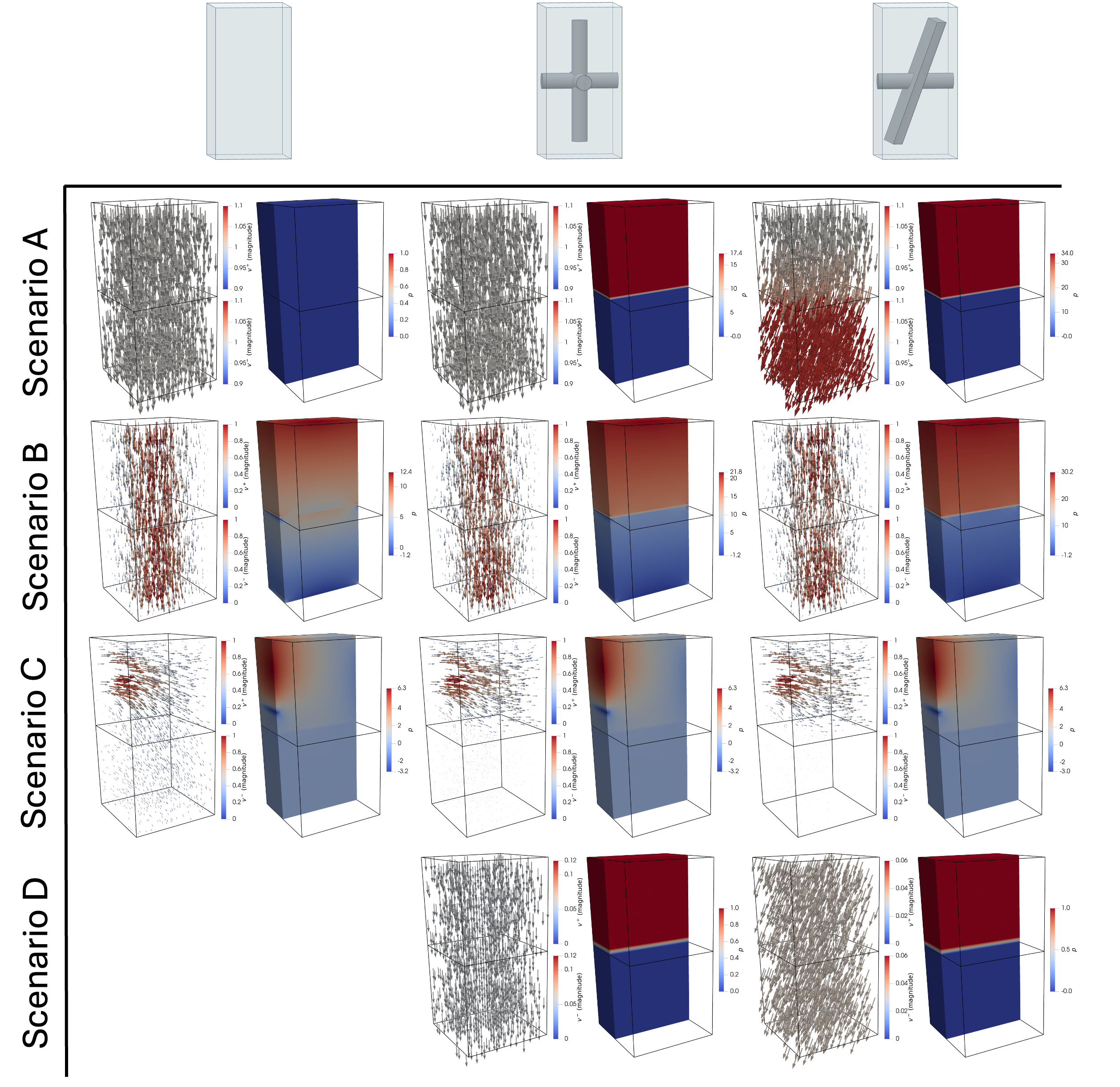}
    \caption{Simulation results of the macroscopic model for Scenarios A, B, C, and D for effective coefficients arising from a cell without obstacle, a symmetric cell (involving the cylindrical cross with h = 1 and r = 0.1), and an asymmetric cell (containing an inclined segment). Each row of figures corresponds to one simulation scenario and each column corresponds to one cell geometry. For each combination of scenario and cell geometry, both the velocity and the pressure are visualized next to each other. The problem combining Scenario D and the cell without obstacle is not well-posed, see also Remark~\ref{rem:well-posedness}.}
\label{fig:simulation_results_macroscopic_model}
\end{figure}
The numerical solution for all combinations of scenarios and cell geometries are illustrated in \cref{fig:simulation_results_macroscopic_model}. Here each row of figures corresponds to one simulation scenario and each column corresponds to one cell geometry. For all combinations, the macroscopic velocities and pressures are presented next to each other including both, the solution in the upper bulk, $\Omega^+$, and the solution in the lower bulk, $\Omega^-$. We note that, in some scenarios, the pressure exhibits a discontinuity at the interface $\Sigma$, and that the discrete ansatz space used for the simulations explicitly allows for such a discontinuity. The seemingly continuous transition of the pressure at the interface $\Sigma$, which in \cref{fig:simulation_results_macroscopic_model} is visible as a white line, is solely an artifact of ParaView’s interpolation of the discontinuous pressure field. 
\par 
For Scenario A, the macroscopic velocity is not affected by the inclusion of the cylindrical cross obstacle in the cell when compared to the velocity obtained from the simulation based on the cell without a solid phase. Meanwhile, the pressure, that appeared to be constant across $\Omega$ in the simulation with obstacle-free cell, is only piecewise constant after inclusion of the cross-shaped obstacle with a discontinuity across the interface $\Sigma$. This observation is shared on a qualitative basis with the simulations based on the cell with the inclined plane, although there the pressure drop across the interface is more severe. Additionally, the macroscopic velocity profile now shows a deflection that is consistent with the orientation of the inclined segment in the cell.

We emphasize that for the symmetric cell geometry, the solution of the transmission problem is given by a constant velocity and a piecewise constant pressure (constant in every bulk domain), namely
\begin{align}\label{eq:sol_ScenarioA}
v^{\pm} = - e_3,  \quad p^+ = 2K_{33}^+ = 2 K_{33}^- \quad \mbox{and} \quad  p^- = 0.
\end{align}
In fact, the only slightly critical relations that need to be checked are the transmission conditions \eqref{eq:macro_Tnormal} and \eqref{eq:macro_Ttangential}, that is, the jump condition for the normal component of the normal stress across $\Sigma$ and the tangential stress condition. For the latter, we use $K_{31}^{\pm} = K_{32}^\pm =0$, see \eqref{eq:rel_eff_coef_1}, to obtain $[K^{\pm} e_3]_t = 0$, implying that \eqref{eq:macro_Ttangential} is valid.
For the constant fluid velocity $v^\pm = -e_3$, the condition \eqref{eq:macro_Tnormal} reads:
\begin{align*}
    p^+ - p^- = K^+(-e_3)\cdot (-e_3) - K^- (-e_3)\cdot e_3 = K^+_{33} + K^-_{33} = 2 K_{33}^+,
\end{align*} 
which is satisfied by $p^\pm$ from 
\eqref{eq:sol_ScenarioA}. We observe that the numerical simulations in \cref{fig:simulation_results_macroscopic_model} are in good agreement with this explicit solution.
\par
In Scenario B, the presence of an obstacle again leads to a discontinuous pressure that jumps at the interface. The parabolic inflow profile is essentially maintained across the domain for all cell geometries, only a slight widening and narrowing of the velocity profile around the interface can be spotted. Moreover, the deflection of the flow is not visible in the case of the cell with the inclined plane.
\par 
In Scenario C, the parabolic inflow profile prescribes a velocity parallel to the interface. Nonetheless, it can be seen that a flux across the interface is established, in particular when the case of the obstacle-free cell is considered. A pronounced local pressure drop can be observed along the line where the prescribed parabolic inflow profile vanishes. This local pressure variation can be attributed to the abrupt change in the boundary velocity gradient at the edge of the inflow region, which induces large local stresses and a corresponding pressure response.
\par
Scenario D is characterized by a pure stress boundary condition that drives the flow. As mentioned in Remark~\ref{rem:well-posedness}, the transmission model derived from the cell without an obstacle is not well-posed for this scenario. Therefore, in \cref{fig:simulation_results_macroscopic_model} this case does not appear. For the symmetric cell geometry, we can, as in Scenario A, compute the explicit solution, namely 
\begin{align}\label{eq:sol_ScenarioD}
v^\pm = -\frac{1}{2 K_{33}}e_3, \quad p^+=1, \quad p^-=0,
\end{align}
and we see that here, too, the numerical results agree well with the explicit solution.
For the asymmetric case of the inclined plane, the velocity is approximately constant throughout the entire domain and has the direction $v^\pm/|v^\pm|\approx(0,-0.39,-0.92)^T$, which is close to the inclination of the plane in the cell, which is oriented approximately along $(0,-0.57,-0.82)$. The qualitative behavior is therefore consistent, and the remaining deviation can likely be attributed to the additional cylindrical obstacle and to the fact that the inclined plane does not extend over the full height of the cell.
\par
Taken together, these observations allow to draw several conclusions. First, we see that an asymmetry in the geometry of the obstacle within the cell causes a deflection of the flow. This is most clearly visible in Scenarios A and D, but the effect is also evident in the other scenarios.  Furthermore, pressure discontinuities across the interface are observed in several scenarios. The pressure jump is particularly pronounced when a normal volume flux across the interface is enforced through Dirichlet boundary conditions. Conversely, the interface can significantly reduce the velocity if the flow through the interface is not enforced by the physics of the boundary conditions (Scenarios C and D).

\subsection{Parameter study with respect to the parameter \texorpdfstring{$h$}{h}}
In this section, we examine how a change of the obstacle geometry within the periodicity cell $Z$ affects the solutions to the macroscopic model \eqref{prob:macroscopic}, while considering a cell geometry generated by the cylindrical cross, see \cref{fig:PeriodicityCell-sym_asym}(b). To this end, we compute numerical solutions to the four scenarios introduced in the previous section. Hereby, we gradually increase the height $h$ of the cylinder aligned with the $y_3$-axis from $1.85$ to $1.99$ in increments of $0.01$, that is, until it almost touches the top and bottom boundary of $\Omega$, and evaluate the norm of the velocity as well as the value of the pressure along the $x_3$-axis, see \cref{fig:solution_along_z_axis}.
\par
In Scenario A, the velocity appears to be independent of the height of the cylinder in the cell; however, the pressure jump across the interface increases as $h$ increases. There is a particularly strong increase in the jump for values of $h$ close to $2$, while the increase of the jump for smaller values of $h$ is rather mild. 
Scenarios B, C and D show a decrease of the velocity at the interface for increasing $h$, and similarly to Scenario A, the pressure jump increases with increasing $h$ in Scenario B and C. In Scenario D, the pressure jump appears to be independent of the value of $h$. It is noteworthy that the system reacts in a somewhat complementary manner to the increase of $h$ in Scenarios A and D. In A, the flow rate through the domain is prescribed by the \textit{Dirichlet} boundary condition at the top boundary, resulting in the velocity being constant and the pressure jump increasing for increasing $h$. In contrast, the velocity decreases while the pressure jump remains constant for $h$ approaching $2$ in Scenario D, where a non-zero boundary load is prescribed at the top boundary via a \textit{Neumann} boundary condition.
\par 
\begin{figure}
    \centering
    \includegraphics[width=1.0\linewidth]{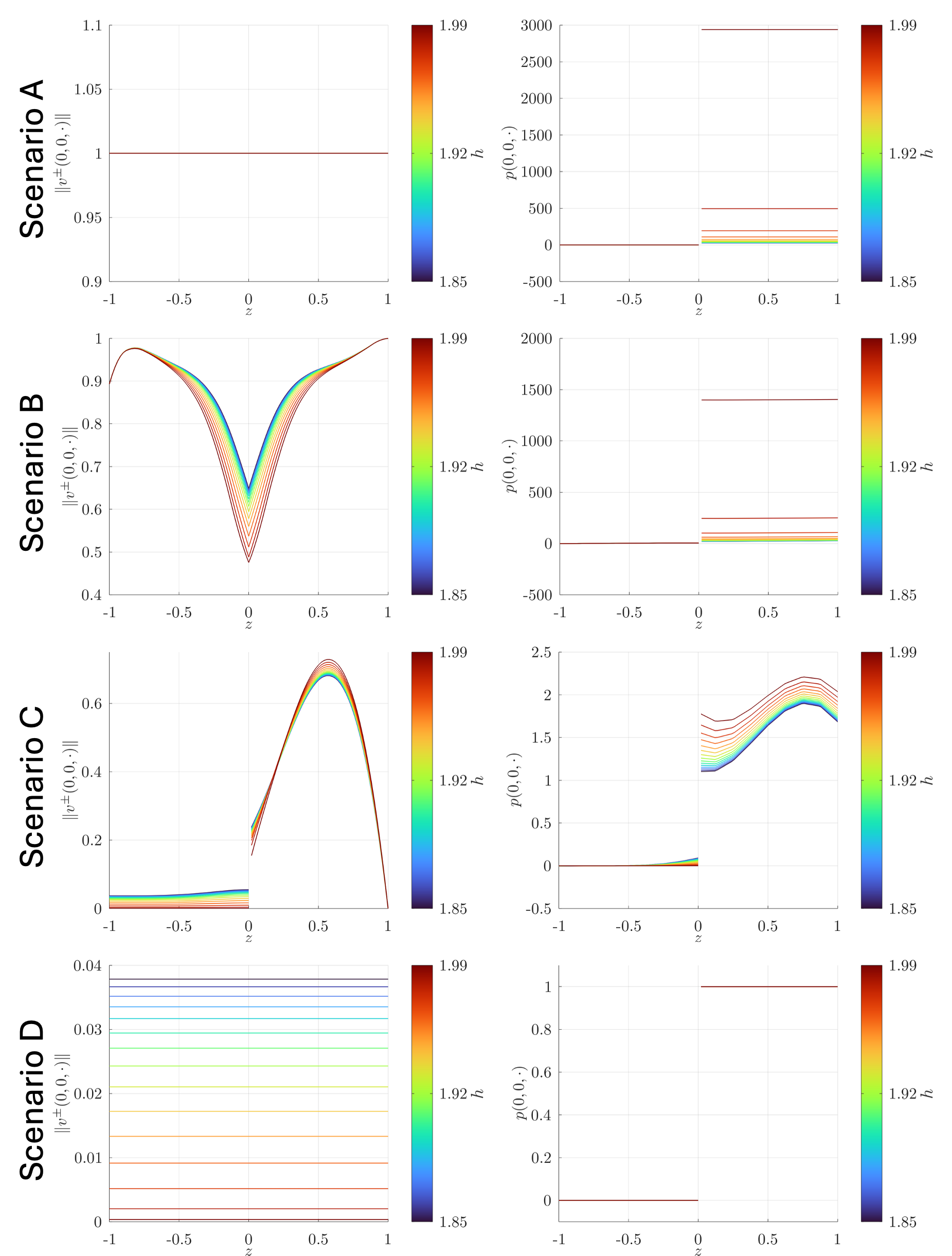}
    \caption{Parameter study for increasing height $h$ of the $y_3$ cylinder in the cylindrical cross obstacle in the cell. For $h \in [1.85, 1.99]$ and $r=0.1$, the macroscopic velocity magnitude and the pressure are evaluated along the $x_3$ axis in the macroscopic domain $\Omega$ for simulation Scenarios A, B, C, and D. The colors of the curves correspond to the values of the parameter $h$.}
    \label{fig:solution_along_z_axis}
\end{figure}

In the following, we look more closer to Scenario D, and perform a quantitative study of the effective velocity (for the case of the symmetric cell geometry). As we can see from \Cref{fig:scenario_D_h_study}(a), when the cylinder height $h$ is increased from 1.85 towards $1.99$, the velocity magnitude at the origin decreases monotonically and eventually approaches zero. The data points follow a slightly sigmoidal curve, indicating a nonlinear dependence of the macroscopic velocity on the cylinder height. The relation between the  effective velocity and the effective coefficient $K_{33} := K^\pm_{33}$ is made explicit in \Cref{fig:scenario_D_h_study}(b). Here, the values of $K_{33}$ corresponding to $h \in [1.85, 1.99]$ are plotted together with the effective velocity norm at the origin in a log-log graph. Over the entire investigated parameter range, the computed velocity values lie almost exactly on the reference curve $f(K_{33}) = \frac{1}{2K_{33}}$. Hence, the numerical results exhibit the relationship 
\begin{align}\label{eq:rel_v_eff_K33}
\| v(0,0,0) \| \approx \frac{1}{2K_{33}}.
\end{align}
In the double-logarithmic representation, this manifests itself as a straight line with slope $-1$. Relation \eqref{eq:rel_v_eff_K33} is validated by the explicit solution given in \eqref{eq:sol_ScenarioD}. Furthermore, the results in \Cref{fig:scenario_D_h_study} are also consistent with the findings about the dependence of the coefficient $K_{33}$ on the parameter $h$, see \cref{fig:Coeff_Sim_Res}.
\begin{figure}[H]
    \centering
    \subfloat[]{
        \includegraphics[width=0.49\linewidth]{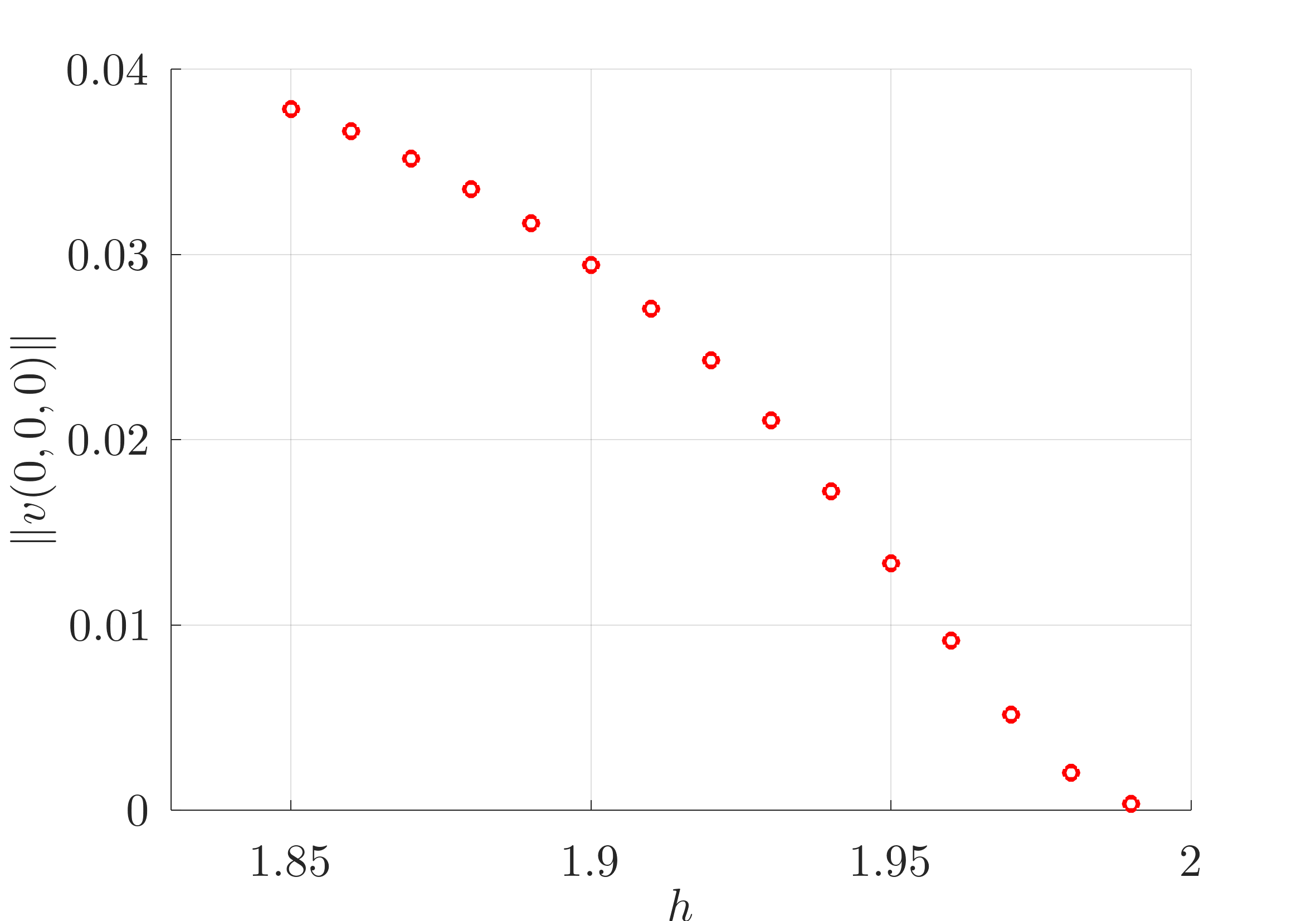}}\hfill
    \subfloat[]{
        \includegraphics[width=0.49\linewidth]{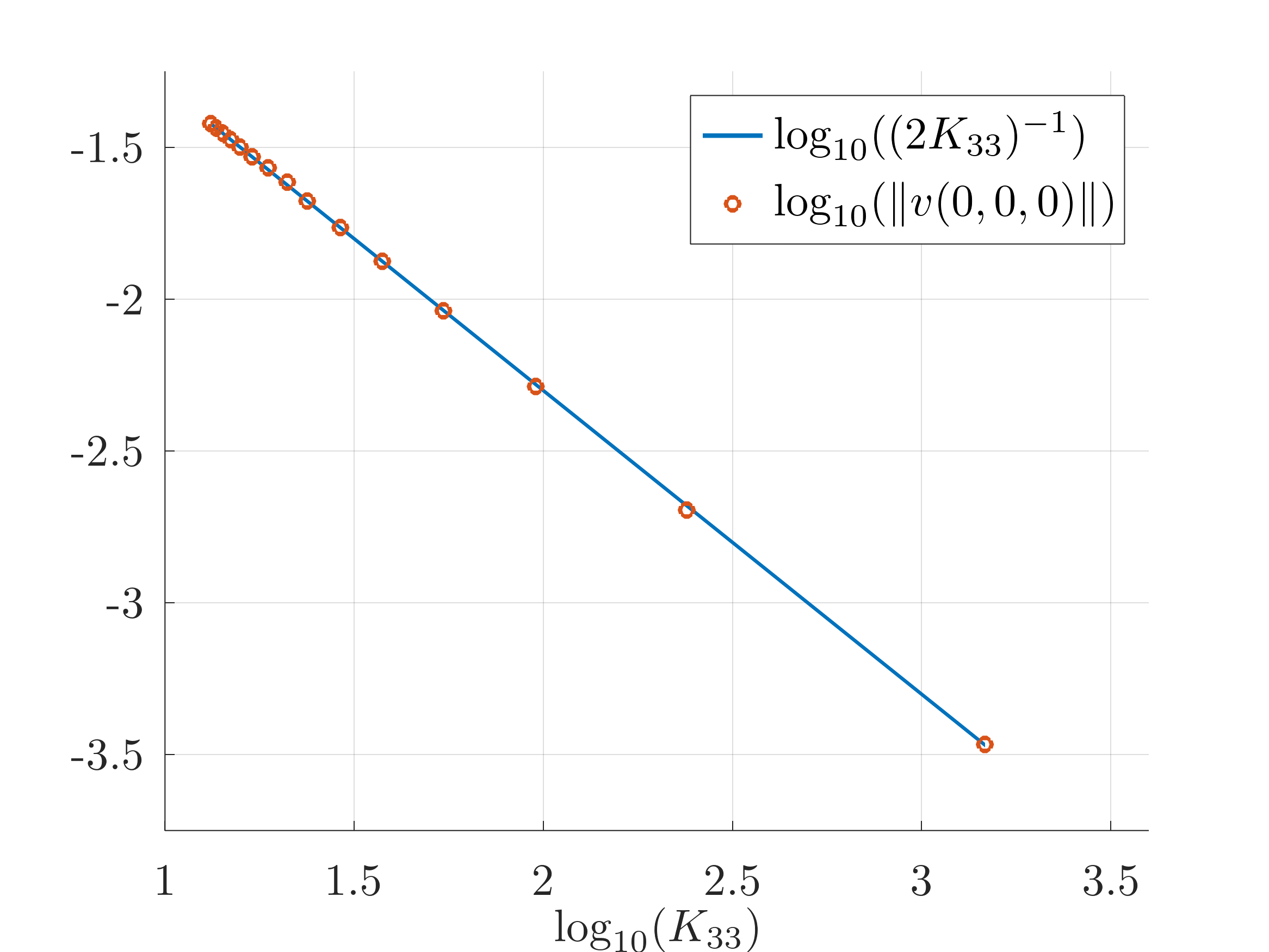}}\hfill
\caption{Detailed study of the macroscopic velocity in Scenario D for different heights $h \in [1.85, 1.99]$ of the $y_3$ cylinder in the cylindrical cross obstacle in the cell. Left: The norm of the macroscopic velocity evaluated at the origin for different values of $h$. Right: Log-log plot of the macroscopic velocity norm over $K_{33}:= K^\pm_{33}$ computed for different values of $h$ together with the curve corresponding to $f(K_{33}) = \frac{1}{2K_{33}}$.} 
\label{fig:scenario_D_h_study}
\end{figure}
\begin{remark} \label{remark_divergence}
In \cite[Section 6.3]{NeussRaduGahn2025}, the model \eqref{prob:macroscopic} together with the boundary conditions from Scenario D was considered for the case $h=2$, and the effective velocity at the interface $\Sigma$ was found to be zero. Thus, by our parameter study in  \cref{fig:scenario_D_h_study}, we show numerically the continuous dependence of the velocity on the parameter $h$ (i.e., on the geometry of the microstructure), in the limit $h\to 2$. We mention that this continuity result is not provided by the homogenization results from \cite{NeussRaduGahn2025}, since  
for values of $h$ approaching $2$ the assumptions on the geometry of the periodicity cell $Z$ are violated: a small scale $2-h$ given by the distance between $Z_s$ and $S^\pm$ appears.
\end{remark}
\subsection{
Numerical study of the approximation accuracy by comparison with the microscopic model}
As mentioned previously, the model \eqref{prob:macroscopic} is an approximation of fluid flow between two bulk domains through a thin porous layer (see \cref{fig:IntroFig}). This was obtained by homogenization and dimension reduction methods in \cite{NeussRaduGahn2025} and \cite{GahnNeussRadu2026}. In this section, we introduce in detail the microscopic $\eps$-model corresponding to the effective model from Scenario A, see Section~\ref{sec:simul_transm_model}. As long as $\varepsilon$ is not too small, it is feasible to generate finite element meshes in \texttt{gmsh} that resolve the microstructure and result in manageable computational cost. This allows us to numerically compare microscopic and effective solutions and provides a numerical justification of the effective transmission model. 

We start by introducing the microscopic geometry:
Assume $\eps^{-1} \in \mathbb{N}$. The domain $\Omega$ consists of the bulk domains $\Oepm$ and the thin layer $\oeps^M$ defined by
\begin{align*}
\oeps^+ := \Sigma \times (\eps,H),  \qquad \oeps^- := \Sigma \times (-H,-\eps), 
\qquad \oeps^M := \Sigma \times (-\eps,\eps).
\end{align*}
The thin layer $\oeps^M$ consists of a fluid and solid introduced in the following.
Let $K_{\eps}:= \{k \in \mathbb{Z}^2 \times \{0\} \, : \, \eps(Z + k) \subset \oeps^M\}$ and   define the fluid/solid part of the thin layer together with its interface 
by
\begin{align*}
\oeps^{M,f} &:= \mathrm{int} \left( \bigcup_{k \in K_{\eps}} \eps \left(\overline{Z_f} + k \right) \right), 
\\
\oeps^{M,s} &:= \mathrm{int} \left( \bigcup_{k \in K_{\eps}} \eps \left(\overline{Z_s} + k \right) \right),
\\
\Geps &:= \mathrm{int}\left( \overline{\oeps^{M,s}} \cap \overline{\oeps^{M,f}}\right).
\end{align*}
The microscopic geometry within $\oeps^M$ is given by periodically arranged cylindrical crosses obtained from the standard cell in \cref{fig:PeriodicityCell-sym_asym}(b).
The interface between the bulk domains $\oeps^{\pm}$ and the thin layer $\oeps^{M}$ is denoted by $S_{\eps}^{\pm} = \Sigma \times \{\pm \eps\}$.
Now, the whole fluid part of the microscopic domain is denoted by
\begin{align*}
    \oef := \Omega \setminus \overline{\oeps^{M,s}}.
\end{align*}

We consider the microscopic incompressible Stokes model for the fluid velocity $\veps$ and the fluid pressure $\peps$ (where we denote by $(\veps^\alpha, \peps^\alpha), \, \alpha \in \{+,-,M\}$, the restrictions of the solution $(\veps, \peps)$ to the corresponding subdomains $\oeps^\alpha$): 
\begin{subequations}\label{MicroscopicModelStationary}
\label{def:micro_equations_fluid_pde_stationary}
\begin{align}
- \nabla \cdot \sigma_{\eps}(\veps,\peps)  &= 0
&\mbox{ in }& \oeps^f,
\\
\nabla \cdot \veps &= 0, &\mbox{ in }& \oef,
\\
\veps^\pm &= \veps^M &\mbox{ on }& S_{\eps}^{\pm}
\\
\sigma_{\eps}(\veps^\pm,\peps^\pm)\cdot \nu & = \sigma_{\eps}(\veps^M,\peps^M)\cdot \nu &\mbox{ on }& S_{\eps}^{\pm},
\\
\veps &= (0,0,-1)^T &\mbox{ on }& \Gamma^+_\text{cap},
\\
\sigma_{\eps}(\veps,\peps)\cdot \nu &= 0
&\mbox{ on }& \Gamma^-_\text{cap},
\\
\label{MicroModelBCDirichletStationary}
\veps &= 0
&\mbox{ on }&  \Geps
\\
\veps,\, \peps &\quad\Sigma\mbox{-periodic},
\end{align}
where the stress tensor $\sigma_{\eps}(\veps,\peps)$ is defined by
\begin{align}
    \sigma_{\eps}(\veps,\peps) \,:=
    \begin{cases}
    \sigma_{\eps}(\veps^\pm,\peps^\pm)
    \\
    \sigma_{\eps}(\veps^M,\peps^M)
    \end{cases}:=
    \begin{cases}
        D(\veps^{\pm}) - \peps^{\pm}I & \mbox{ in } \oeps^{\pm}
        \\
        \eps D(\veps^M) - \eps^{-1} \peps^M I &\mbox{ in } \oeps^{M,f}.
    \end{cases}
\end{align}
\end{subequations}

\cref{fig:microscopicSolutions} shows the microscopic velocity $v_\varepsilon$ and the pressure $p_\varepsilon$ for $\varepsilon = \frac12, \frac14, \frac16, \frac18$. One can observe that the velocity is locally increased in the pores of the membrane. The pressure has a discontinuity at the interface $S_\eps^+$ between the upper bulk domain $\Omega^+_\varepsilon$  and the thin layer $\oeps^M$. At the interface $S^-_\varepsilon$ between the layer and the lower bulk domain the pressure appears to be continuous. Notably, this observation is reversed if the boundary conditions at the upper and lower cap of the domain $\Gamma^\pm_\text{cap}$ are interchanged, whereas the pressure has a discontinuity at both interfaces $S^\pm_\varepsilon$ if the \textit{Dirichlet} conditions $v^\pm_\varepsilon = -e_3$ are imposed on $\Gamma^\pm_\text{cap}$ (not shown).
 \begin{figure}
    \centering
    \subfloat[$\varepsilon = \frac{1}{2}$]{
        \includegraphics[width=0.49\linewidth]{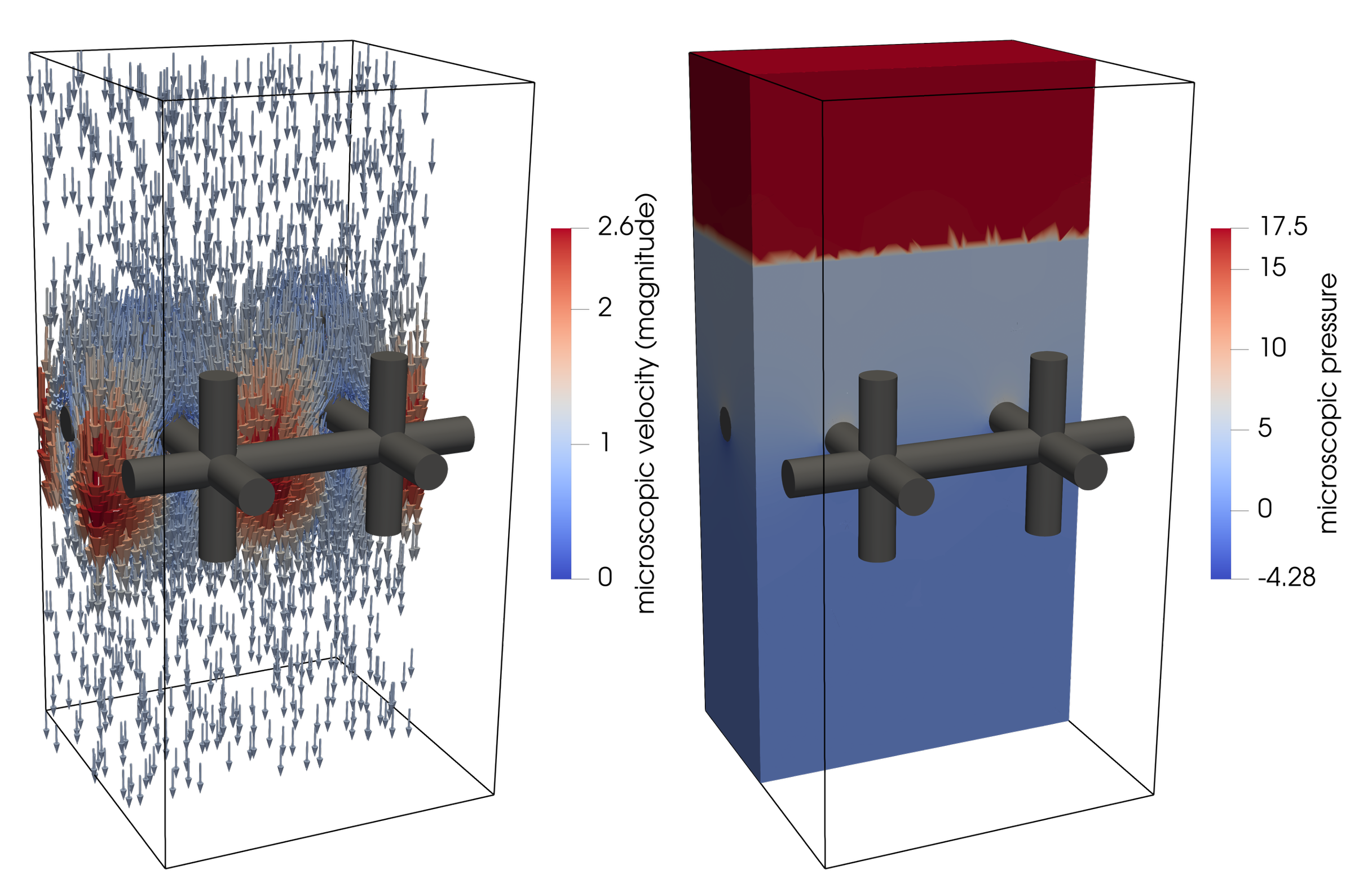}}\hfill
    \subfloat[$\varepsilon = \frac{1}{4}$]{
        \includegraphics[width=0.49\linewidth]{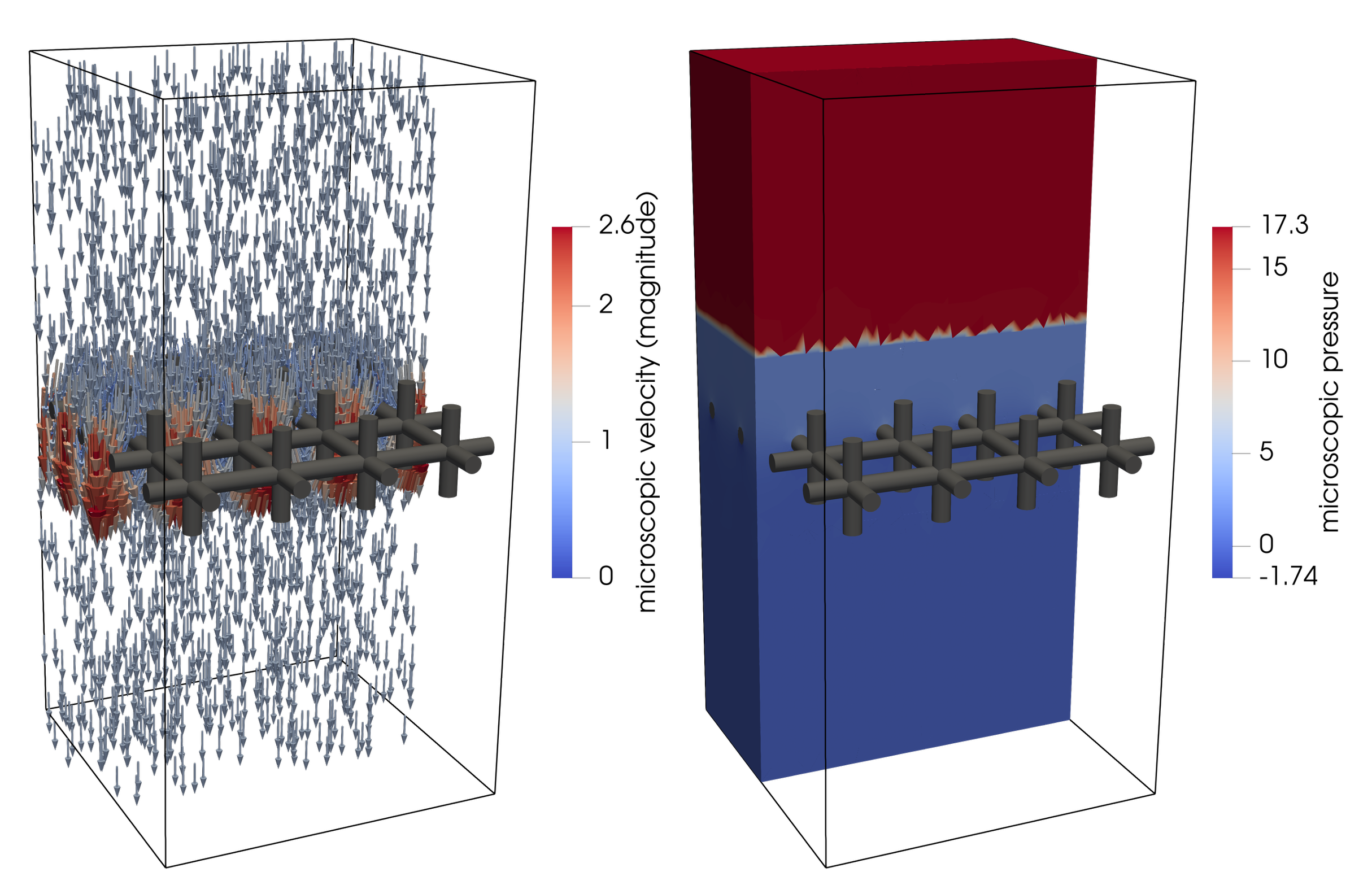}}\hfill
    \subfloat[$\varepsilon = \frac{1}{6}$]{
        \includegraphics[width=0.49\linewidth]{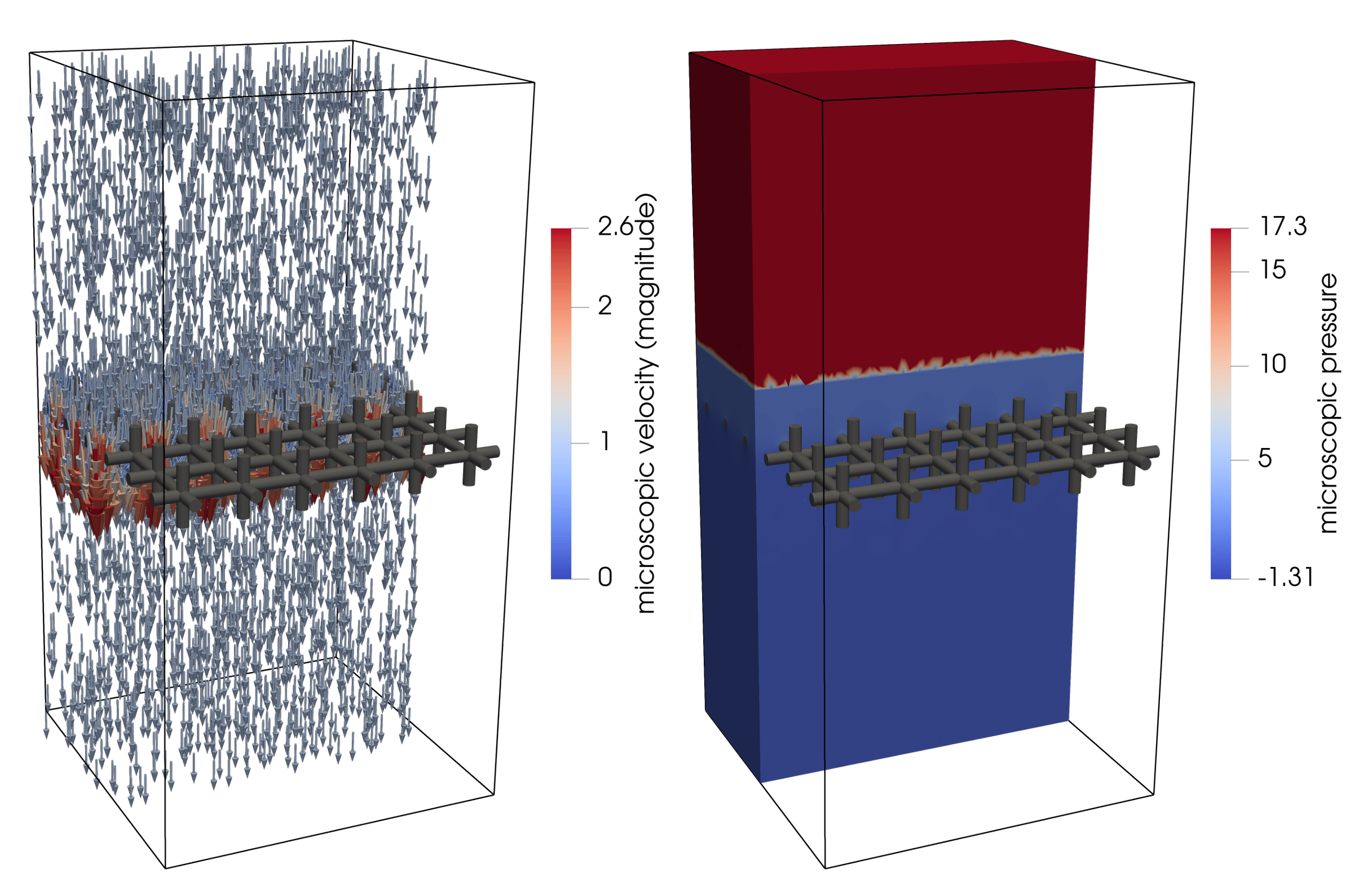}}\hfill
    \subfloat[$\varepsilon = \frac{1}{8}$]{
        \includegraphics[width=0.49\linewidth]{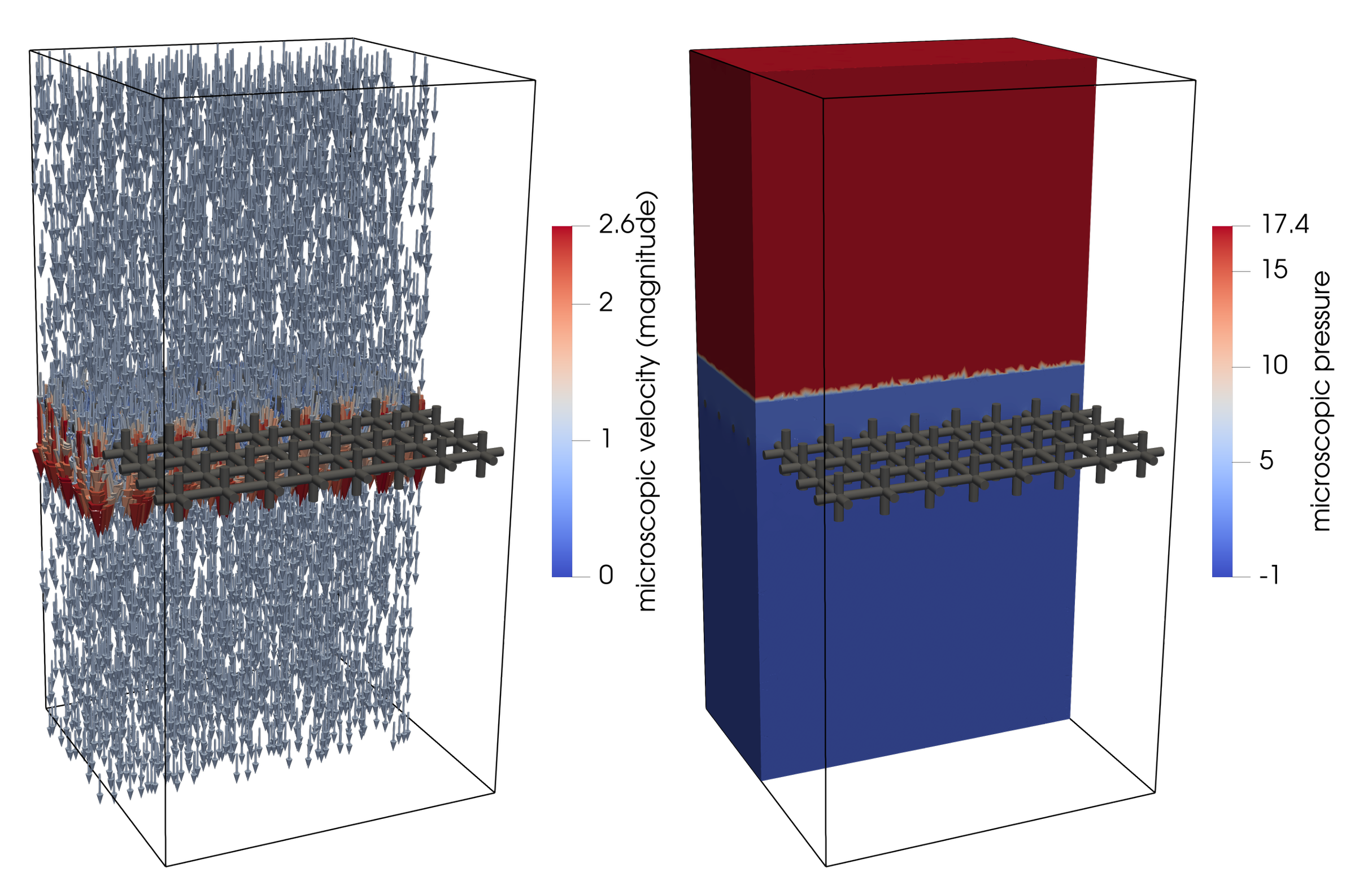}}\hfill
\caption{Numerical solutions $v_\varepsilon$ and $p_\varepsilon$ to the microscopic model \eqref{def:micro_equations_fluid_pde_stationary} for decreasing values of $\varepsilon$.} 
\label{fig:microscopicSolutions}
\end{figure}%
We now compare these microscopic solutions with the effective solution shown in \cref{fig:simulation_results_macroscopic_model} in the first row, third and forth column. At a first glance we can see a good qualitative match, especially at small values of $\eps$. To illustrate more quantitatively how the microscopic solutions are approximated by the effective one, the third component of the velocity as well as the value of the pressure are evaluated along the $x_3$-axis in \cref{fig:convergencePlots}. In (a), the magnitude of the negative microscopic velocity peak at the microstructure remains (almost) constant for decreasing $\varepsilon$ while, however, the width of the peak shrinks. In (b), one can see that the location of the pressure discontinuity approaches the location of the discontinuity of the effective solution at $x_3 = 0$.
\begin{figure}
    \centering
    \subfloat[]{
        \includegraphics[width=0.49\linewidth]{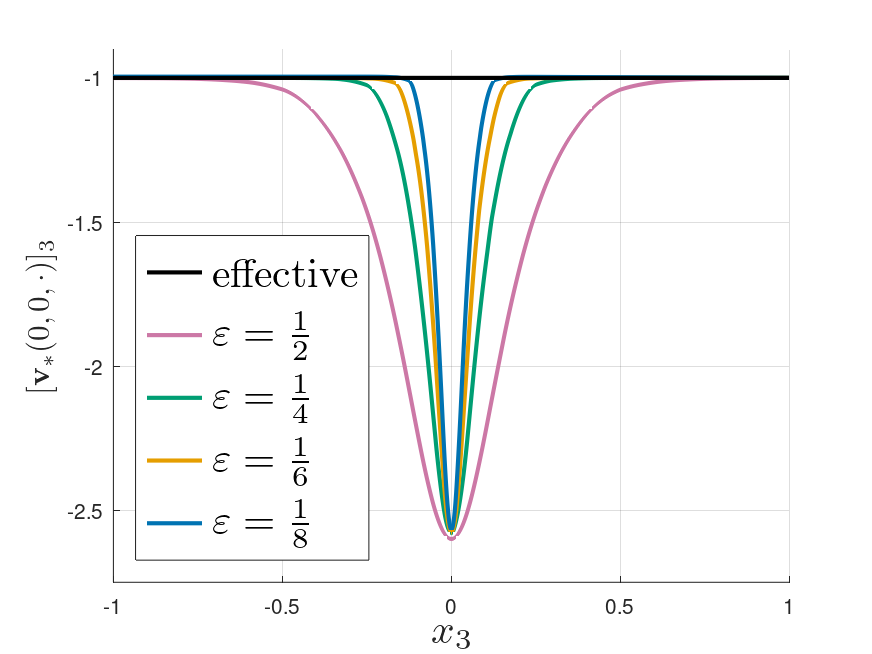}}\hfill
    \subfloat[]{
        \includegraphics[width=0.49\linewidth]{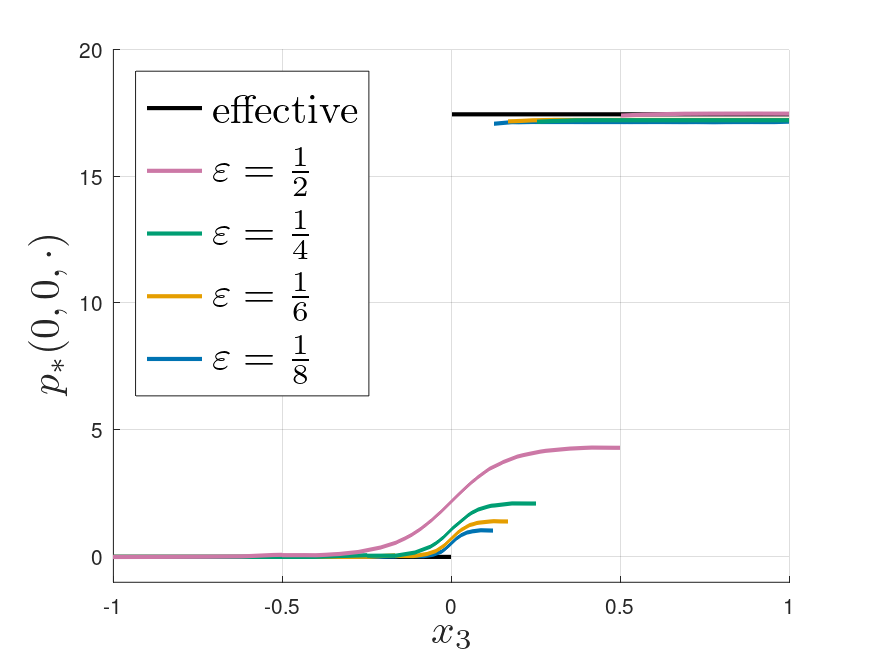}}\hfill
\caption{Comparison of the numerical solution to the effective problem and to the microscopic problems for decreasing values of $\varepsilon$. The lines illustrate the third component of the velocity as well as the pressure, evaluated along the line $(0,0,x_3)$, $ x_3 \in [-1,1]$.} 
\label{fig:convergencePlots}
\end{figure}
\section{Conclusion and outlook}
\label{sec:conclusionoutlook}
In this paper, we provided numerical methods for the analysis and implementation of effective transmission models  with coefficients computed from microscopic cell solutions for fluid flow through porous membranes. The numerical simulations offered new insights about the localization and symmetry properties of cell solutions, the structure of effective coefficients and the impact of the membrane microstructure on the behavior of the macroscopic (effective) solution. 

On the microscale, we considered the case of a rigid solid. For many applications, in particular in biological tissues, it is essential to account for the elasticity of the solid phase, which leads to an additional fluid–structure interaction at the effective interface $\Sigma$. Corresponding effective interface models have been rigorously derived in \cite{gahn2025beffective}, and their numerical treatment is part of our ongoing work. A related model was considered in \cite{KrierOrlikPanasenkoSteiner2024}, where the deformation of the effective interface is described by a Kirchhoff–Love plate equation and the fluid transport through the plate is taken into account by means of a phenomenological Darcy-type interface law. A crucial difference from the model derived in \cite{gahn2025beffective}, besides the continuity of the tangential fluid velocity assumed in \cite{KrierOrlikPanasenkoSteiner2024}, concerns the effective coefficients governing the fluid transport and its coupling to the elastic displacement. While the Darcy-type transmission law in \cite{KrierOrlikPanasenkoSteiner2024} involves the classical permeability tensor, the coefficients in \cite{gahn2025beffective} are obtained from cell problems specifically adapted to the underlying microscopic fluid–structure interaction and enter more general Navier-slip-type interface conditions. Consequently, the numerical methods developed in the present paper have to be extended to account for this more complex fluid–structure interaction and the associated generalized cell problems.

\section*{Acknowledgments}
The first author acknowledges the funding by the Deutsche Forschungsgemeinschaft (DFG, German Research Fundation) within the research project DyNano (Dynamics and Control of Superparamagnetic Iron Oxide Nanoparticles in Simple and Branched Vessels) - Project-ID 518492286.

\end{document}